\documentclass[10pt,reqno]{amsart}

\usepackage{enumerate}
\usepackage{amsmath, amssymb, amsthm}
\usepackage{mathrsfs}
\usepackage{esint}
\usepackage{xcolor}
\usepackage{mathtools}
\usepackage{hyperref}
\usepackage{bm}
\usepackage{accents}
\usepackage[foot]{amsaddr}

\makeatletter

\numberwithin{equation}{section}

\newtheorem{thm}{Theorem}[section]
\newtheorem{theorem}[thm]{Theorem}
\newtheorem{lemma}[thm]{Lemma}
\newtheorem{corollary}[thm]{Corollary}
\newtheorem{lem}[thm]{Lemma}

\theoremstyle{definition}

\newtheorem{definition}[thm]{Definition}

\newtheorem{assumption}[thm]{Assumption}
\theoremstyle{remark}

\newtheorem{remark}[thm]{\bf{Remark}}

\newcommand\bE{\mathbb{E}}

\newcommand\bH{\mathbb{H}}

\newcommand\bL{\mathbb{L}}

\newcommand\bN{\mathbb{N}}

\newcommand\bP{\mathbb{P}}

\newcommand\bR{\mathbb{R}}

\newcommand\bZ{\mathbb{Z}}

\newcommand\cF{\mathcal{F}}

\newcommand\frH{\mathfrak{H}}

\newcommand{\mysection}[1]{\section{#1}}

\begin{document}

\title[Nonlocal equations in weighted Sobolev spaces]{Nonlocal elliptic and parabolic equations with general stable operators in weighted Sobolev spaces}

\author{Hongjie Dong$^{1}$}
\address{$^1$ Division of Applied Mathematics, Brown University, 182 George Street, Providence, RI 02912, USA}
\email{Hongjie\_Dong@brown.edu}
\thanks{H. Dong was partially supported by Simons Fellows Award 007638 and the NSF under agreement DMS-2055244.}

\author{Junhee Ryu$^{1,2}$}
\address{$^2$ Department of Mathematics, Korea University, 145 Anam-ro, Seongbuk-gu, Seoul,
02841, Republic of Korea}
\email{Junhee\_Ryu@brown.edu}
\thanks{J. Ryu was supported by the National Research Foundation of Korea (NRF) grant funded by the Korea government (MSIT) (No. NRF-2020R1A2C1A01003354)}



\begin{abstract}
We study nonlocal elliptic and parabolic equations on $C^{1,\tau}$ open sets in weighted Sobolev spaces, where $\tau\in (0,1)$.
The operators we consider are infinitesimal generators of symmetric stable L\'evy processes, whose L\'evy measures are allowed to be very singular. Additionally, for parabolic equations, the measures are assumed to be merely measurable in the time variable.
\end{abstract}

\maketitle

\mysection{Introduction}

In this paper, we study the parabolic equation
\begin{equation*}
\begin{cases}
\partial_t u(t,x)=L_t u(t,x)+f(t,x),\quad &(t,x)\in(0,T)\times D,
\\
u(0,x)=0,\quad & x\in D,
\\
u(t,x)=u_0(x), \quad & (t,x)\in(0,T)\times D^c,
\end{cases}
\end{equation*}
as well as the corresponding elliptic equation
\begin{equation} \label{intro_e}
\begin{cases}
Lu(x)=f(x),\quad &x\in D,
\\
u(x)=0, \quad & x\in D^c.
\end{cases}
\end{equation}
where $D$ is a $C^{1,\tau}$ open set with $\tau\in(0,1)$. Here, $L_t$ is a time-dependent symmetric nonlocal operator of order $\alpha\in(0,2)$, while $L$ is independent of the time variable.
More specifically, the operators $L_t$ is defined by
\begin{align} \label{op_rep}
L_t u(x) = \frac{1}{2}\int_{\bR^d} \left( u(x+y)+u(x-y)-2u(x) \right)\, \nu_t(dy), \quad t\in(0,T),
\end{align}
where $\nu_t$ is a nondegenerate $\alpha$-stable symmetric L\'evy measure for each $t\in(0,\infty)$. The operator $L$ for the elliptic equation is defined as \eqref{op_rep} with time-independent L\'evy measure $\nu$ instead of $\nu_t$;
\begin{equation} \label{op}
Lu(x):=\frac{1}{2}\int_{\bR^d} \left( u(x+y)+u(x-y)-2u(x) \right) \,\nu(dy).
\end{equation}

For instance, the operator $L$ becomes the fractional Laplacian $-(-\Delta)^{\alpha/2}$ when $\nu(dy)=c|y|^{-d-\alpha}dy$ for some $c>0$. Another simple example is the generator of $d$ independent one dimensional symmetric stable L\'evy processes,
\begin{equation*}
  Lu(x):=\frac{1}{2}\sum_{i=1}^d \int_{\bR^d} \frac{f(x+y_ie_i)+f(x-y_ie_i)-2f(x)}{|y_i|^{1+\alpha}}\,dy_i,
\end{equation*}
where $e_i$ is the unit vector in the $i$-th coordinate. In this case, the spectral measure of L\'evy measure (see \eqref{eq9050110}) is a sum of $2d$ Dirac measures defined on the unit sphere, which is very singular.

These types of operators can be derived as the infinitesimal generator of (time-inhomogeneous) L\'evy processes. Such stochastic processes have been widely studied in both analysis and probability theory, and appeared in various fields such as physics and mathematical finance.

In \cite{RS14}, it was shown that the optimal regularity of solution to \eqref{intro_e} is $C^{\alpha/2}(D)$, not $C^{\alpha}(D)$, even in the case when $L=-(-\Delta)^{\alpha/2}$ and $f\in L_\infty(D)$. Therefore, obtaining solvability for the equations in $H_p^{\alpha}(D)$, the (unweighted) Sobolev spaces, may not be possible. This necessitates the exploration of Sobolev spaces with weights.

In this paper, we are mainly concerned with the weighted Sobolev spaces $H_{p,\theta}^\gamma(D)$ and $L_p((0,T);H_{p,\theta}^\gamma(D))$.
In particular, when $\gamma\in \bN$,
\begin{equation*}
  \|u\|_{H_{p,\theta}^\gamma(D)}:=\left(\sum_{k=0}^\gamma \int_{D} |d_x^k D_x^ku|^p d_x^{\theta-d}\, dx \right)^{1/p}.
\end{equation*}
Here, $d_x$ denotes the distance from $x$ to $D^c$ and the powers of $d_x$ are used to control the behavior of $u$ and its derivatives near the boundary.
These spaces were presented in \cite[Section 2.6.3]{ML68} for the specific case $p=2$ and $\theta=d$. They were generalized in a unified manner for $p\in(1,\infty)$ and $\theta,\gamma\in\bR$ in \cite{Krylovhalf} in order to establish an $L_p$-theory of stochastic partial differential equations (SPDEs). See, for instance, \cite{K94,KLline,KLspace}. Since the work \cite{Krylovhalf}, there are many results on second-order equations in the weighted Sobolev spaces. See, for instance, \cite{DK15,KKL14,KK2004,KLee13,KN09,S23}.

The purpose of this paper is to present maximal regularity of solutions to nonlocal equations in such weighted Sobolev spaces. In particular, for the elliptic equation \eqref{intro_e}, we prove
\begin{equation} \label{ineq9061602}
  \int_D \left(|d_x^{-\alpha/2}u|^p + |d_x^{\alpha/2}(-\Delta)^{\alpha/2}u|^p \right) d_x^{\theta-d}\, dx \leq N \int_D |d_x^{\alpha/2} f|^{p} d_x^{\theta-d}\, dx
\end{equation}
provided that $\theta$ is in a certain range.
Here, due to the presence of $d_x^{\alpha/2}$, both $f$ and $(-\Delta)^{\alpha/2}u$ are in a space that allows them to blow up near the boundary.
 In \cite{Dirichlet,KR23}, weighted estimates similar to \eqref{ineq9061602} were proved when $L=-(-\Delta)^{\alpha/2}$, $D$ is a $C^{1,1}$ open set, and $\theta$ is in the sharp range $(d-1,d-1+p)$.

Compared to the results \cite{Dirichlet,KR23}, we study the equations in a more general setting as we consider operators with highly singular L\'evy measures as well as $C^{1,\tau}$ open sets with any $\tau\in (0,1)$.  More precisely, when $D$ is a half space or a bounded $C^{1,\tau}$ convex domain, under certain ellipticity condition we establish \eqref{ineq9061602} where $\theta$ is in the optimal range $(d-1,d-1+p)$. For general $C^{1,\tau}$ open sets, the same results are obtained when $\theta\in(d-\alpha/2,d-\alpha/2+\alpha p/2)$. Here, the range of $\theta$ is restricted since we deal with singular L\'evy measures. Nevertheless, in the case of convex domains, such constraints are not necessary.
  Regarding the parabolic equations, the operators $L_t$ are assumed to be merely measurable in the time variable, and a parabolic version of \eqref{ineq9061602} is obtained.

For the proof of the main results, we derive a priori estimates and then use the method of continuity. We first prove zeroth order estimates in Section \ref{sec_zero}. For second-order equations, these weighted estimates can be obtained by using integration by parts and the product rule of the differentiation. See \cite[Section 6]{Krylovhalf}. However, it appears that these fundamental methods are not directly applicable to nonlocal operators. In \cite{Dirichlet}, zeroth order estimates for $L=-(-\Delta)^{\alpha/2}$ were obtained by using the sharp heat kernel estimates for the fractional Laplacian on $C^{1,1}$ open sets. Compared to this, we do not rely on the representation in terms of the fundamental solution since it is not available under our assumptions. Our approach is more elementary and can be applied to a larger class of nonlocal operators.
Next, in Section \ref{sec_high}, we provide higher order regularity of solutions by using an estimate of the commutator term. See Lemma \ref{lem6091055}. In this subsection, we apply the $L_p$-maximal regularity of equations with time-dependent operators in the whole space, which is briefly handled in Appendix \ref{sec_whole} by appealing a result in \cite{Mikul Cauchy}. It is worth noting that no regularity of open sets is utilized in Section \ref{sec_high}.
Finally, to apply the method of continuity, the solvability of equations for $L=-(-\Delta)^{\alpha/2}$ is presented at the beginning of Section \ref{sec_proof}.

Now we give a short review on other relevant work. We first refer the reader to \cite{AG23,G14,Gfrac,G18,G23} for $L_p$-maximal regularity results of equations with pseudodifferential operators satisfying the $\mu$-transmission condition. In particular, in \cite{G14}, it was proved that if $f\in L_p(D)$ and $D$ is $C^{\infty}$, then \eqref{intro_e} has a unique solution in the $\mu$-transmission space. These results were extended to $C^{1,\tau}$ open sets with $\tau>\alpha$ in \cite{AG23,G23}. In a similar setting, parabolic equations were handled in \cite{G18}. We remark that in this paper, we introduce a different approach and consider open sets with lower regularity.
For interior regularity results, we refer the reader to \cite{BWZ17,BWZ18,C17,N20}. For instance, in \cite{BWZ17}, it was proved that if $f\in L_p(D)$, then a solution $u$ to \eqref{intro_e} with $L=-(-\Delta)^{\alpha/2}$ is in $H_{p,\text{loc}}^\alpha(D)$. See also \cite{BGPR20,FKV15,GH22,GHS23,HJ96} for results on $L_2$ spaces.

Let us also mention related results in H\"older spaces. In \cite{RS14}, it was proved that when $L=-(-\Delta)^{\alpha/2}$, $D$ is a $C^{1,1}$ bounded domain, and $f\in L_\infty(D)$, any solution $u$ to \eqref{intro_e} satisfies $d_x^{-\alpha/2}u\in C^{\delta}(D)$ for some $\delta>0$. This result was generalized in \cite{ABC16,fernandez2017regularity,ros2016regularity} by considering more general operators. See also \cite{RSdom,RV16} for the results for equations on $C^{1,\tau}$ or less regular domains. In \cite{DRSV22}, it was shown that if the operator is nonsymmetric, then boundary behavior of solution is more complicated, while for symmetric operators, all solutions behave like a fixed power of the distance function $d_x^{\alpha/2}$.
See also \cite{CS,RS16} for results about nonlinear nonlocal equations. Particularly, boundary behavior of solutions to fully nonlinear equations were investigated in \cite{RS16}.

We finish the introduction by introducing the notation used in this paper.
 We use $``:=''$ or $``=:''$ to denote a definition. For a real number $a\in \bR$, we write $a_+:=\max\{a,0\}$. For any $D\subset \bR^d$, $d_x:=\text{dist}(x,D^c)$. By $\bN$ and $\bZ$, we denote the set of natural numbers and the set of integers, respectively. We denote $\bN_0:=\bN\cup\{0\}$. As usual, $\bR^d$ stands for the Euclidean space of points $x=(x_1,\dots,x_d)$. We also denote
$$
B_r(x):=\{y\in\bR^d : |x-y|<r\}, \quad \bR^{d}_+=\{(x_1,\dots,x^d)\in\bR : x_1>0\}.
$$
We write $\bR:=\bR^1$ and $\bR_+:=\bR_+^1$.
We use $D^n_x u$ to denote the partial derivatives of order $n\in\bN_0$ with respect to the space variables. By $C^2(\bR^d)$, we denote the space of twice continuously differentiable functions on $\bR^d$.
By $C_b^2(\bR^d)$, we denote the space
of functions whose derivatives of order up to $2$ are bounded and continuous.
 For $1<p<\infty$, $0<T\leq\infty$, and a Banach space $B$, $L_p((0,T);B)$ denotes the set of $B$-valued Lebesgue measurable functions $u$ such that
$$
\|u\|_{L_p((0,T);B)}:=\left(\int_0^T |u|_B^p\, dt\right)^{1/p}.
$$
For Borel measures $l_1$ and $l_2$ on $\bR^d$, we write $l_1\leq l_2$ if
$$
l_1(A)\leq l_2(A)
$$
for any Borel set $A\subset \bR^d$.
Lastly, we use the convention that negative powers of $0$ is defined as $0$.

\mysection{Main results}

We first introduce the assumptions for the operators.
Let $\nu$ be a L\'evy measure on $\bR^d$, that is, $\nu$ is a $\sigma$-finite (positive) measure on $\bR^d$ such that $\nu(\{0\})=0$, and
$$
\int_{\bR^d} \min\{1,|y|^2\}\, \nu(dy)<\infty.
$$
A L\'evy measure $\nu$ is said to be symmetric if $\nu(-dx)=\nu(dx)$.
For $\alpha\in(0,2)$, we say that a L\'evy measure $\nu$ is $\alpha$-stable if there is a nonnegative finite measure $\mu$ on the unit sphere $S^{d-1}$, called the spherical part of $\nu$, such that
\begin{equation} \label{eq9050110}
  \nu(A)= \int_{S^{d-1}} \int_0^\infty 1_{A}(r\theta) \frac{dr}{r^{1+\alpha}} \,\mu(d\theta).
\end{equation}
In particular, when $d=1$ and $\nu(dy)=|y|^{-1-\alpha}dy$, we have
\begin{equation} \label{opd1}
Lu(x)=\frac{1}{2}\int_{-\infty}^\infty \left(u(x+y)+u(x-y)-2u(x) \right) \frac{dy}{|y|^{1+\alpha}}=-\pi(-\Delta)^{\alpha/2}u.
\end{equation}

We first state our assumption on $\nu$.   We say that $\alpha$-stable L\'evy measure $\nu$ is nondegenerate if $\nu$ satisfies the following assumption.
  \begin{assumption} \label{assum}
    $(i)$ There exists $\lambda>0$ such that
  \begin{align} \label{nonde}
\lambda\leq\inf_{\rho\in S^{d-1}} \int_{S^{d-1}} |\rho\cdot\theta|^\alpha \,\mu(d\theta),
  \end{align}
 where $\mu$ is the spherical part of $\nu$.

    $(ii)$ There exists $\Lambda>0$ such that
      \begin{align} \label{bound}
    \int_{S^{d-1}} \mu( d\theta)\leq \Lambda <\infty.
  \end{align}
  \end{assumption}

  Next we consider time-dependent L\'evy measure $\nu_t$ in \eqref{op_rep}. Assume that $(\nu_t)_{0<t<T}$ is a family of $\alpha$-stable symmetric L\'evy measure, that is,
\begin{equation*}
  \nu_t(A)= \int_{S^{d-1}} \int_0^\infty 1_{A}(r\theta) \frac{dr}{r^{1+\alpha}}\,\mu_t(d\theta),
\end{equation*}
where $\mu_t$ is the spherical part of $\nu_t$ for each $t$. Here is our assumption on $\nu_t$.

\begin{assumption} \label{assum_t}
  $(i)$ If $f$ is integrable with respect to $\nu_t$ for all $t\in(0,T)$, then the mapping
  $$
t\rightarrow \int_{\bR^d} f(y)\,\nu_t(dy)
  $$
  is measurable.

  $(ii)$ There exist $\lambda>0$ and nondegenerate symmetric $\alpha$-stable L\'evy measure $\nu^{(1)}$ such that
    \begin{align} \label{ineq9040040}
    \nu_t\geq \nu^{(1)}, \quad \forall t\in(0,T),
  \end{align}
  and
  \begin{align*}
\lambda\leq \inf_{\rho\in S^{d-1}} \int_{S^{d-1}} |\rho\cdot\theta|^\alpha \mu^{(1)} (d\theta),
  \end{align*}
 where $\mu^{(1)}$ is the spherical part of $\nu^{(1)}$.

    $(iii)$ There exists $\Lambda>0$ such that
      \begin{align*}
    \int_{S^{d-1}} \mu_t( d\theta)\leq \Lambda <\infty, \quad t\in(0,T),
  \end{align*}
  where $\mu_t$ is the spherical part of $\nu_t$.
\end{assumption}

\begin{remark}
Obviously, if $\nu_t$ satisfies Assumption \ref{assum_t}, then for each $t\in(0,T)$, $\nu_t$ satisfies Assumption \ref{assum}.
\end{remark}

Now we present some weighted $L_p$ spaces which will be used throughout this paper.
Let $D$ be an open set with nonempty boundary and $d_x:=\text{dist} (x, D^c)$. For any  $p>1, \theta\in \bR$, and $n\in \bN_0$, we denote weighted Sobolev spaces of nonnegative integer orders by
\begin{equation*}
  H^{n}_{p,\theta}(D):=\{u: u, d_x Du, \cdots, d_x^n D^nu \in L_{p,\theta}(D)\},
\end{equation*}
   where $L_{p,\theta}(D)$ is an $L_p$ space with the measure $d_x^{\theta-d}$ on $D$. The norm in this space is defined as
\begin{align} \label{eq6111713}
\|u\|_{H_{p,\theta}^{n}(D)}=\sum_{|\beta|\leq n} \left( \int_{D}|d_x^{|\beta|}D_x^{\beta}u(x)|^p d_x^{\theta-d} \, dx\right)^{1/p}.
\end{align}
Here, notice that $L_{p,\theta}(D)=H_{p,\theta}^0(D)$.

 Next, we generalize these spaces to Sobolev and Besov spaces of arbitrary order.
Let $\{\zeta_n\}_{n\in\bZ}$ be a collection of nonnegative functions in $C^{\infty} (D)$ with the following properties:
\begin{align}
    &\label{zeta_1}(i)\,\,\text{supp} (\zeta_n) \subset \{x\in D : c_1e^{-n}< d_x <c_2e^{-n}\}, \quad c_2>c_1>0,
    \\
    &\label{zeta_2}(ii)\,\,\sup_{x\in\bR^d}|D^m_x \zeta_n (x)| \leq N_me^{mn},\quad \forall m\in\bN_0
    \\
    &\label{zeta_3}(iii)\,\,\sum_{n\in\bZ} \zeta_n(x) > c>0,\quad\forall x\in D.
\end{align}
To construct such functions, one can take, for example, mollifications of indicator functions of the sets of the type $ \{x\in D : c_3e^{-n}< d_x<c_4e^{-n}\}$.

Let $H_p^{\gamma}$ and $B_{p,p}^\gamma$ denote the Bessel potential space and the Besov space on $\bR^d$, respectively.
For any  $p\in(1,\infty), \theta\in \bR$ and $\gamma\in\bR$, weighted Sobolev spaces $H_{p,\theta}^{\gamma}(D)$ and weighted Besov spaces $B_{p,p;\theta}^{\gamma}(D)$ are defined as collections of distributions $u$ on $D$ such that
\begin{equation} \label{defSob}
\|u\|_{H_{p,\theta}^{\gamma}(D)}^p:=\sum_{n\in\bZ}e^{n\theta}\|\zeta_{-n}(e^n\cdot)u(e^n\cdot)\|_{H_p^{\gamma}}^p<\infty,
\end{equation}
and
\begin{equation*}
\|u\|_{B_{p,p;\theta}^{\gamma}(D)}^p:=\sum_{n\in\bZ}e^{n\theta}\|\zeta_{-n}(e^n\cdot)u(e^n\cdot)\|_{B_{p,p}^\gamma}^p<\infty,
\end{equation*}
respectively.
The spaces $H_{p,\theta}^{\gamma}$ and $B_{p,p;\theta}^{\gamma}$ are independent of choice of $\{\zeta_n\}$. See, for instance, \cite[Proposition 2.2]{Lototsky}. More specifically, if $\{\zeta_{n}\}$ satisfies \eqref{zeta_1} and \eqref{zeta_2}, then we have
  \begin{equation*}
\sum_{n\in\bZ} e^{n\theta} \|\zeta_{-n}(e^n\cdot)u(e^n\cdot) \|_{H_p^\gamma}^p \leq N \|u\|^p_{H_{p,\theta}^{\gamma}(D)}.
  \end{equation*}
  Also, the reverse inequality holds if $\{\zeta_{n}\}$ additionally satisfies \eqref{zeta_3}.
  Furthermore, when $\gamma=n \in \bN_0$, the two norms \eqref{eq6111713}  and \eqref{defSob} are equivalent. The similar properties also hold for $B_{p,p;\theta}^\gamma(D)$.

Take an infinitely differentiable function $\psi$ in $D$ such that $N^{-1}\psi\leq d_x \leq N\psi$ and for any $m\in \bN_0$
\begin{align*}
\sup_{D} |d_x^m D^{m+1}_x \psi(x)|\leq N(m)<\infty.
\end{align*}
For instance, we can take $\psi(x):=\sum_{n\in\bZ} e^{-n} \zeta_n(x)$ (see also e.g. \cite{KK2004}).

Below we collect some facts about the space $H^{\gamma}_{p,\theta}(D)$. For $\nu\in \bR$, we write $u\in \psi^{-\nu} H_{p,\theta}^{\gamma}(D)$ if $\psi^{\nu} u \in  H_{p,\theta}^{\gamma}(D)$.

\begin{lemma}\label{lem_prop}
Let $D$ be an open set with nonempty boundary, $\gamma,\theta \in \bR$ and $1<p<\infty$.

(i) The space $C^{\infty}_c(D)$ is dense in both $H^{\gamma}_{p,\theta}(D)$ and $B^{\gamma}_{p,p;\theta}(D)$.

(ii) For $\nu\in\bR$, we have
$$
H_{p,\theta}^\gamma (D) = \psi^{\nu} H_{p,\theta+\nu p}^\gamma(D)\quad \text{and}\quad
B_{p,p;\theta}^\gamma (D) = \psi^{\nu} B_{p,p;\theta+\nu p}^\gamma(D).
$$
Moreover,
\begin{align*}
N^{-1} \|\psi^{-\nu} u\|_{H_{p,\theta+\nu p}^\gamma(D)} \leq \|u\|_{H_{p,\theta}^\gamma (D)} \leq N \|\psi^{-\nu} u\|_{H_{p,\theta+\nu p}^\gamma(D)},
\end{align*}
and
\begin{align*}
N^{-1} \|\psi^{-\nu} u\|_{B_{p,p;\theta+\nu p}^\gamma(D)} \leq \|u\|_{B_{p,p;\theta}^\gamma (D)} \leq N \|\psi^{-\nu} u\|_{B_{p,p;\theta+\nu p}^\gamma(D)},
\end{align*}
where $N$ depends only on $d,\gamma,\nu,p,$ and $\theta$.

(iii) (Duality) Let
\begin{equation*}
1/p+1/p'=1, \quad \theta/p+\theta'/p' = d.
\end{equation*}
Then, $H_{p',\theta'}^{-\gamma}(D)$ and $B_{p',p';\theta'}^{-\gamma}(D)$ are the dual spaces of $H_{p,\theta}^\gamma(D)$ and $B_{p,p;\theta}^\gamma(D)$, respectively.
\end{lemma}

\begin{proof}
We only deal with $H_{p,\theta}^\gamma(D)$ since the proofs for $B_{p,p;\theta}^\gamma(D)$ are similar to those for $H_{p,\theta}^\gamma(D)$.
When $D=\bR_+^d$, all the claims are proved by Krylov in \cite{KrylovSome}, and those are generalized by Lototsky in \cite{Lototsky} for arbitrary domains. See Proposition 2.2 and 2.4, and Theorem 4.1 in \cite{Lototsky}, whose proofs are still valid for general open sets.
 The lemma is proved.
\end{proof}

Next, we introduce solution spaces for the parabolic equation. For $\gamma,\theta\in\bR$, $p\in(1,\infty)$, and $T\in(0,\infty]$, we denote
\begin{equation*}
  \bH_{p,\theta}^\gamma(D,T):=L_p((0,T);H_{p,\theta}^\gamma(D)), \quad \bL_{p,\theta}(D,T):=L_p((0,T);L_{p,\theta}(D)).
\end{equation*}
Let $\alpha\in(0,2)$. For $u\in\psi^{\alpha/2}\bH_{p,\theta}^{\alpha}(D,T)$ with $u(0,\cdot)\in \psi^{\alpha/2-\alpha/p}B_{p,p;\theta}^{\alpha-\alpha/p}(D)$, we say $u\in \frH_{p,\theta}^{\alpha}(D,T)$ if there exists $f\in \psi^{-\alpha/2}\bL_{p,\theta} (D,T)$ such that for any $\phi\in C_c^\infty(D)$
\begin{equation*}
\langle u(t,\cdot),\phi\rangle_D=\langle u(0,\cdot),\phi\rangle_D + \int_0^t \langle f(s,\cdot),\phi\rangle_D \,ds, \quad \forall \, t\in(0,T),
\end{equation*}
where $\langle\cdot,\cdot\rangle$ is defined as
$$
\langle f,g \rangle_E:=\int_E fg \,dx,
$$
where $f$ and $g$ are measurable functions defined on $E\subset \bR^d$.
Here, we write $\partial_t u:=f$. The norm in this space is defined as
\begin{align*}
\|u\|_{\frH_{p,\theta}^{\alpha}(D,T)} :&= \|\psi^{-\alpha/2} u\|_{\bH_{p,\theta}^{\alpha}(D,T)} + \|\psi^{\alpha/2} \partial_t u\|_{\bL_{p,\theta}(D,T)}
\\
&\quad+ \|\psi^{-\alpha/2+\alpha/p}u(0,\cdot)\|_{B_{p,p;\theta}^{\alpha-\alpha/p}(D)}.
\end{align*}

We introduce a notion of weak solution.

\begin{definition} \label{def6110010}

Let $T\in(0,\infty]$ and $D\subset\bR^d$ be an open set.

$(i)$ (Parabolic problem) Given $u_0 \in L_{1,\text{loc}}(D)$ and $f \in L_{1,\text{loc}}((0,T)\times D)$,
we say that $u$ is a (very weak) solution to the equation
\begin{equation} \label{main_para}
\begin{cases}
\partial_t u(t,x)=L_tu(t,x)+f(t,x),\quad &(t,x)\in(0,T)\times D,
\\
u(0,x)=u_0,\quad & x\in D,
\\
u(t,x)=0, \quad & (t,x)\in(0,T)\times D^c,
\end{cases}
\end{equation}
if (a) $u=0 \ a.e.$ in $(0,T)\times D^c$, (b) $\langle u(t,\cdot), \phi \rangle_{\bR^d}$ and $ \langle u(t,\cdot), L_t\phi \rangle_{\bR^d}$ exist for any $t< T$ and $\phi\in C_c^\infty(D)$, and (c) for $\phi\in C^{\infty}_c(D)$ the equality
\begin{align*}
\langle u(t,\cdot),\phi \rangle_{\bR^d}= \langle u_0,\phi \rangle_{D} + \int_0^t \langle u(s,\cdot),L_s\phi \rangle_{\bR^d} \,ds + \int_0^t \langle f(s,\cdot),\phi \rangle_D \,ds
\end{align*}
holds for all $t<T$.

$(ii)$ (Elliptic problem) Given $f \in L_{1,\text{loc}}(D)$,
we say that $u$ is a (very weak) solution to the equation
\begin{equation} \label{main_ell}
\begin{cases}
Lu(x)=f(x),\quad &x\in D,
\\
u(x)=0, \quad & x\in D^c.
\end{cases}
\end{equation}
if (a) $u=0 \ a.e.$ in $D^c$, (b) $\langle u, \phi\rangle_{\bR^d}$ and $ \langle u, L\phi \rangle_{\bR^d}$ exist for $\phi\in C_c^\infty(D)$, and (c)  for $\phi\in C^{\infty}_c(D)$ we have
\begin{equation*}
\langle u,L\phi \rangle_{\bR^d} = \langle f,\phi \rangle_D.
\end{equation*}
\end{definition}
We remark here that if $u$ is a sufficiently regular strong solution, it is also a (very weak) solution in the sense of Definition \ref{def6110010}.

The main purpose of this paper is to derive weighted maximal $L_p$ estimates in $C^{1,\tau}$ open sets. Below we give the formal definition of $C^{1,\tau}$ open sets.

\begin{definition}
For $\tau\in (0,1)$, an open set $D\subset\bR^d$ is said to be a $C^{1,\tau}$ open set if there exists a constant $R_0>0$ such that for any $x_0\in \partial D$, there is a $C^{1,\tau}$ function $\Phi:\bR^{d-1}\to\bR$ and a coordinate system $y=(y_1,y')$ centered at $x_0$, in which
\begin{equation*}
  D\cap B_{R_0}(x_0)=B_{R_0}(0)\cap \{y:y_1>\Phi(y')\}.
\end{equation*}
\end{definition}

Now we state the main result for the parabolic equations.
\begin{theorem}[Parabolic case] \label{thm_para}
  Let $p\in(1,\infty)$, $\alpha\in(0,2)$, $\tau\in(0,1)$, and $T\in(0,\infty]$. Suppose that $\nu_t$ satisfies Assumption \ref{assum_t}. Assume that $\theta\in(d-1,d-1+p)$ if $D$ is a half space or a bounded $C^{1,\tau}$ convex domain, and $\theta\in(d-\alpha/2,d-\alpha/2+\alpha p/2)$ if $D$ is a bounded $C^{1,\tau}$ open set. Then, for any $f\in \psi^{-\alpha/2}\bL_{p,\theta}(D,T)$ and $u_0\in \psi^{\alpha/2-\alpha/p}B_{p,p;\theta}^{\alpha-\alpha/p}(D)$, there is a unique weak solution $u$ to \eqref{main_para} such that $u\in \frH_{p,\theta}^\alpha(D,T)$, and for this solution we have
  \begin{align} \label{apriori}
\|u\|_{\frH_{p,\theta}^{\alpha}(D,T)} \leq N \left(\|\psi^{\alpha/2}f\|_{\bL_{p,\theta}(D,T)}+ \|\psi^{-\alpha/2+\alpha/p}u_0\|_{B_{p,p;\theta}^{\alpha-\alpha/p}(D)} \right),
  \end{align}
  where $N=N(d,p,\alpha,\theta,\tau,D,\lambda,\Lambda)$.
\end{theorem}

Here is the main result for the elliptic equations.
\begin{theorem}[Elliptic case] \label{thm_ell}
  Let $p\in(1,\infty)$, $\alpha\in(0,2)$, and $\tau\in(0,1)$. Suppose that $\nu$ satisfies Assumption \ref{assum}.  Assume that $\theta\in(d-1,d-1+p)$ if $D$ is a half space or a bounded $C^{1,\tau}$ convex domain, and $\theta\in(d-\alpha/2,d-\alpha/2+\alpha p/2)$ if $D$ is a bounded $C^{1,\tau}$ open set. Then, for any $f\in \psi^{-\alpha/2}L_{p,\theta}(D)$, there is a unique weak solution $u$ to \eqref{main_ell} such that $u\in \psi^{\alpha/2}H_{p,\theta}^\alpha(D)$, and for this solution we have
  \begin{align} \label{aprioriell}
\|\psi^{-\alpha/2}u\|_{H_{p,\theta}^{\alpha}(D)} \leq N \|\psi^{\alpha/2}f\|_{L_{p,\theta}(D)},
  \end{align}
  where $N=N(d,p,\alpha,\theta,\tau,D,\lambda,\Lambda)$.
\end{theorem}

\mysection{Analysis of distance functions}

Throughout this section, Assumption \ref{assum} is enforced.

We first introduce several useful facts on convex domains, which will be used in the proof of Lemma \ref{lem5251636}.
\begin{lem} \label{lem5252044}
Let $D$ be a convex domain with nonempty boundary and $x,y\in \overline{D}$.

(i) For $z:=(1-t)x+ty$ and $t\in[0,1]$, we have
\begin{equation} \label{eq5251655}
  d_z\geq (1-t)d_x +td_y.
\end{equation}
In particular,
\begin{equation*}
  d_z\geq \min\{d_x,d_y\}.
\end{equation*}

(ii) Let $\theta\in S^{d-1}$. Then, if there is no $r_0>0$ such that $x+r_0\theta \in \partial D$, then $d_{x+r\theta}\geq d_x$ for any $r>0$.
\end{lem}

\begin{proof}
$(i)$ Notice that it suffices to prove \eqref{eq5251655} for $x,y\in D$. Let $\hat{z}\in\partial D$ such that $|z-\hat{z}|=d_z$, and $P$ be the hyperplane to $\partial D$ at $\hat{z}$. Take $\hat{x},\hat{y}\in P$ such that both $x-\hat{x}$ and $y-\hat{y}$ are parallel to $z-\hat{z}$. Then, since $D$ is convex, $\hat{x},\hat{y} \in D^c$. Thus,
\begin{equation*}
  d_z=|z-\hat{z}|=(1-t)|x-\hat{x}|+t|y-\hat{y}|\geq (1-t)d_x +td_y.
\end{equation*}

$(ii)$ Assume that $x+r_0\theta\in D$ for all $r_0>0$, and $d_{x+r\theta}<d_x$ for some $r>0$. Take $c\in(0,1)$ such that $d_{x+r\theta}=cd_x$. Then, by \eqref{eq5251655} with $z:=x+r\theta$ and $y:=\frac{1}{1-c}\left(z-cx\right)=x+\frac{r}{1-c}\theta\in D$, we get
\begin{equation*}
  (1-c)d_y \leq d_z -cd_x = 0.
\end{equation*}
However, due to the convexity of $D$, we have $d_y>0$, which gives a contradiction.
The lemma is proved.
\end{proof}

\begin{lem} \label{lem5251636}
Let $D$ be a convex domain with nonempty boundary, $\kappa_0\in(0,1)$, $\kappa_1\in(0,\infty)$, and $\kappa_2\in(-1,0)$. Let $\nu_{\kappa_1}$ be a measure taking the form
\begin{equation} \label{nuform}
  \nu_{\kappa_1}(A)= \int_{S^{d-1}} \int_0^\infty 1_A(r\theta) \frac{dr}{r^{1+\kappa_1}} \,\mu(d\theta).
\end{equation}
Then, for $x\in D$,
\begin{equation} \label{eq5251608}
  \int_{\bR^d} 1_{d_{x+y}\leq \kappa_0 d_x} d_{x+y}^{\kappa_2} \ \nu_{\kappa_1}(dy) \leq N d_x^{-\kappa_1+\kappa_2},
\end{equation}
where $N$ depends only on $d,\kappa_1, \kappa_2, \kappa_0, \Lambda$ and $D$.
\end{lem}

\begin{proof}
  Let $x\in D$ and $\theta\in S^{d-1}$.
To prove \eqref{eq5251608}, by \eqref{bound}, it suffices to show
\begin{equation*}
  \int_{0}^\infty 1_{d_{x+r\theta}\leq \kappa_0 d_x} d_{x+r\theta}^{\kappa_2} r^{-1-\kappa_1} \,dr \leq N d_x^{-\kappa_1+\kappa_2}.
\end{equation*}

  Since $D$ is convex, there exists at most one $r_0\in (0,\infty)$ such that $x+r_0\theta\in\partial D$.
By Lemma \ref{lem5252044} $(ii)$, we only consider the case when such $r_0$ exists. If $d_{x+r\theta}\leq \kappa_0 d_x$, then by \eqref{eq5251655} with $(x+r\theta,x+r_0\theta)$ instead of $(z,y)$, we have $r\geq (1-\kappa_0)r_0$, and
\begin{equation*}
  d_{x+r\theta}\geq \left(1-\frac{r}{r_0}\right)d_x.
\end{equation*}
Since $r_0=|(x+r_0\theta)-x|\geq d_x$ (recall that negative powers of $0$ is defined as $0$),
\begin{align*}
  \int_{0}^\infty 1_{d_{x+r\theta}\leq \kappa_0 d_x} d_{x+r\theta}^{\kappa_2} r^{-1-\kappa_1} \,dr &\leq d_x^{\kappa_2}\int_{(1-\kappa_0)r_0}^{r_0} \left(1-\frac{r}{r_0}\right)^{\kappa_2} r^{-1-\kappa_1} \,dr
  \\
  &\leq N d_x^{\kappa_2} r_0^{-\kappa_1} \leq N d_x^{-\kappa_1+\kappa_2}.
\end{align*}
The lemma is proved.
\end{proof}

Now we present the corresponding result for general open sets without convexity.
\begin{lem} \label{lem5261746}
Let $D$ be an open set with nonempty boundary, $\kappa_0,\kappa_1\in(0,\infty)$, $\kappa_2\in[0,\kappa_1)$, and $\nu_{\kappa_1}$ be a measure taking the form \eqref{nuform}.
For $x\in D$ and $\rho>0$ such that $\kappa_0 d_x\leq \rho$,
\begin{equation} \label{eq5261805}
  \int_{|y|\geq\rho} d_{x+y}^{\kappa_2} \ \nu_{\kappa_1}(dy) \leq N d_x^{-\kappa_1+\kappa_2},
\end{equation}
where $N$ depends only on $d,\kappa_0,\kappa_1, \kappa_2$ and $\Lambda$.
\end{lem}

\begin{proof}
Since $\kappa_2\geq0$, using the relation $d_{x+y}\leq d_x+|y|$, we clearly have
  \begin{align*}
    \int_{|y|\geq\rho} d_{x+y}^{\kappa_2} \ \nu_{\kappa_1}(dy) &\leq N d_x^{\kappa_2} \int_{|y|\geq\rho} \nu_{\kappa_1}(dy) + N \int_{|y|\geq\rho} \nu_{\kappa_1-\kappa_2}(dy)
    \\
    &\leq N d_x^{\kappa_2} \rho^{-\kappa_1} + N \rho^{-\kappa_1+\kappa_2} \leq N d_x^{-\kappa_1+\kappa_2}.
  \end{align*}
  Thus, we have \eqref{eq5261805}. The lemma is proved.
\end{proof}

\begin{corollary} \label{cor5261717}
  Let $D$ be a convex domain with nonempty boundary, $\kappa_1,\kappa_0\in(0,\infty)$, $\kappa_2\in(-1,\kappa_1)$, and $\nu_{\kappa_1}$ be a measure taking the form \eqref{nuform}.
Then, for $x\in D$ and $\rho>0$ such that $\kappa_0 d_x\leq \rho$,
\begin{equation*}
  \int_{|y|\geq \rho} d_{x+y}^{\kappa_2} \ \nu_{\kappa_1}(dy) \leq N d_x^{-\kappa_1+\kappa_2},
\end{equation*}
where $N$ depends only on $d,\kappa_1, \kappa_2, \kappa_0, \Lambda$ and $D$.
\end{corollary}

\begin{proof}
Since the case $\kappa_2\geq0$ is treated in \eqref{eq5261805}, we only consider $\kappa_2<0$. Then, by Lemma \ref{lem5251636}, we have
    \begin{align*}
    &\int_{|y|\geq \rho} d_{x+y}^{\kappa_2} \ \nu_{\kappa_1}(dy)
    \\
    &\leq \int_{\bR^d} 1_{d_{x+y}\leq d_x/2} d_{x+y}^{\kappa_2} \ \nu_{\kappa_1}(dy) + \int_{|y|\geq \rho} 1_{d_{x+y}> d_x/2} d_{x+y}^{\kappa_2} \ \nu_{\kappa_1}(dy)
    \\
    &\leq N d_x^{-\kappa_1+\kappa_2} + N d_x^{\kappa_2} \int_{|y|\geq \rho}\nu_{\kappa_1}(dy)
    \\
    &\leq N d_x^{-\kappa_1+\kappa_2} + N d_x^{\kappa_2} \rho^{-\kappa_1} \leq d_x^{-\kappa_1+\kappa_2}.
  \end{align*}
  The corollary is proved.
\end{proof}

In the rest of this section, we deal with regularized distance functions defined on $C^{1,\tau}$ open sets. We say that
$\widetilde{\psi}$ is a regularized distance function on a $C^{1,\tau}$ open set $D$ if
\begin{equation} \label{eq5282250}
  N^{-1}\widetilde{\psi}(x) \leq d_x \leq N\widetilde{\psi}(x), \quad \widetilde{\psi}\in C^{1,\tau}(\overline{D}), \quad |D^2_x\widetilde{\psi}(x)|\leq Nd_x^{\tau-1}.
\end{equation}
To construct such a function, one can follow the ideas in \cite{GH,KK2004} (see also \cite{D07}).
Since we will compute $L\widetilde{\psi}(x)$ below, we additionally define $\widetilde{\psi}(x)=0$ on $x\in D^c$.

We first state explicit computations for one dimensional operators.
The following two lemmas are extensions of \cite[Propositions 4.3 and 4.4]{DRSV22} (see also \cite{FRobs,Gfrac,J}). The proofs are given in Lemmas \ref{lem_a2} and \ref{lem_a3}.

\begin{lemma} \label{lem_1d_1}
Let $d=1$ and $L$ be an operator of the form \eqref{opd1} with $\alpha\neq1$. Let
\begin{equation*}
u(x):=(x_+)^\beta, \quad \beta\in(-1,\alpha).
\end{equation*}
 Then,
 \begin{equation*}
Lu(x)=K_{\alpha,\beta}(x_+)^{\beta-\alpha}, \quad x\in \bR_+,
 \end{equation*}
 where
  \begin{equation*}
K_{\alpha,\beta}=-\frac{2}{\pi}\Gamma(-\alpha)\Gamma(1+\beta)\Gamma(\alpha-\beta)\cos(\alpha \pi/2)\sin((\beta-\alpha/2)\pi).
 \end{equation*}
 In particular,
 \begin{align} \label{eq6102300}
\begin{cases}
K_{\alpha,\beta}>0, \quad \beta\in(-1,-1+\alpha/2)\cup(\alpha/2,\alpha),
\\
K_{\alpha,\beta}=0, \quad \beta=-1+\alpha/2 \text{ or } \alpha/2,
\\
K_{\alpha,\beta}<0, \quad \beta\in(-1+\alpha/2,\alpha/2).
\end{cases}
 \end{align}
\end{lemma}

\begin{lemma} \label{lem_1d_2}
Let $d=1$ and $L=-(-\Delta)^{1/2}$. Let
\begin{equation*}
u(x):=(x_+)^\beta, \quad \beta\in(-1,1).
\end{equation*}
 Then,
 \begin{equation*}
Lu(x)=K_{1,\beta}(x_+)^{\beta-1}, \quad x\in \bR_+,
 \end{equation*}
 where
  \begin{align*}
K_{1,\beta}=
\begin{cases}
-\beta \cos(\beta \pi), \quad &\beta\in(0,1),
\\
-\frac{1}{\pi}, \quad &\beta=0,
\\
\beta \cos(\beta\pi), \quad &\beta\in(-1,0).
\end{cases}
 \end{align*}
 Moreover, \eqref{eq6102300} still holds true with $\alpha=1$.
\end{lemma}

As a consequence of Lemmas \ref{lem_1d_1} and \ref{lem_1d_2}, we have the following. Here, $K_{\alpha,\beta}$ is taken from the above two lemmas.

\begin{corollary} \label{cor_hs}
Let $d\geq1$ and $L$ be an operator of the form \eqref{op}. Let $\rho\in S^{d-1}$ and define
\begin{equation*}
u(x):=[(x\cdot\rho)_+]^\beta, \quad \beta\in(-1,\alpha).
\end{equation*}
Then, for $t>0$,
\begin{equation*}
  Lu(x)=N_{\alpha,\beta}[(x\cdot\rho)_+]^{\beta-\alpha} \text{ in } \{x\cdot\rho>0\},
\end{equation*}
where
\begin{equation*}
N_{\alpha,\beta}=K_{\alpha,\beta}  \int_{S^{d-1}} |\theta\cdot\rho|^{\alpha} \,\mu(d\theta).
\end{equation*}
In particular,
 \begin{align*}
\begin{cases}
N_{\alpha,\beta}>0, \quad \beta\in(-1,-1+\alpha/2)\cup(\alpha/2,\alpha),
\\
N_{\alpha,\beta}=0, \quad \beta=-1+\alpha/2 \text{ or } \alpha/2,
\\
N_{\alpha,\beta}<0, \quad \beta\in(-1+\alpha/2,\alpha/2).
\end{cases}
 \end{align*}
\end{corollary}

\begin{proof}
Note that for fixed $x\in\bR^d$, $r\in\bR$ and $\theta\in S^{d-1}$,
\begin{equation*}
u(x+r\theta)=v(x\cdot\rho+(\theta\cdot\rho)r).
\end{equation*}
where
\begin{equation*}
v(s):=(s_+)^\beta.
\end{equation*}
Since $r\to u(x+r\theta)$ is a one-dimensional function, by Lemmas \ref{lem_1d_1} and \ref{lem_1d_2},
\begin{align*}
\frac{1}{2}\int_{-\infty}^\infty \left(u(x+r\theta)+u(x-r\theta)-2u(x) \right) \frac{dr}{|r|^{1+\alpha}} &= K_{\alpha,\beta} |\theta\cdot\rho|^{\alpha} (x\cdot\rho)_+^{\beta-\alpha}.
\end{align*}
This equality easily yields the desired result. The corollary is proved.
\end{proof}

The following two lemmas will be used to prove Lemmas \ref{lem_dist_dom} and \ref{lem_dist_open}.
\begin{lem}
Let $0<\varepsilon\leq1$ and $\beta\leq1$. Then, for $a,b>0$,
\begin{equation} \label{eq5261440}
  |a^\beta-b^\beta|\leq N(\beta,\varepsilon)|a-b|^\varepsilon |a^{\beta-\varepsilon}+b^{\beta-\varepsilon}|.
\end{equation}
\end{lem}

\begin{proof}
  Without loss of generality, we only consider the case $0<b<a$. Let $f(x):=|x^\beta-b^\beta|^{1/\varepsilon}$. Then, by the mean value theorem,
  \begin{align} \label{eq5261441}
    f(a)=f(a)-f(b)=(a-b)f'(c),
  \end{align}
  for some $b<c<a$. Here,
  \begin{align*}
    |f'(c)|&\leq \frac{|\beta|}{\varepsilon} |c^\beta-b^\beta|^{1/\varepsilon-1} c^{\beta-1} \leq N (c^{\beta/\varepsilon-\beta}+b^{\beta/\varepsilon-\beta})c^{\beta-1}
    \\
    &= N (c^{\beta/\varepsilon-1} +b^{\beta/\varepsilon-\beta}c^{\beta-1})\leq N (a^{\beta/\varepsilon-1}+b^{\beta/\varepsilon-1}).
  \end{align*}
  This and \eqref{eq5261441} easily yield \eqref{eq5261440}. The lemma is proved.
\end{proof}

\begin{lem} \label{lem8132213}
Let $D$ be a $C^{1,\tau}$ open set, $c\in(0,\infty)$, $\kappa_1\in(0,\infty)$, $\kappa_2\in(-1,\kappa_1)$, and $\nu_{\kappa_1}$ be a measure taking the form \eqref{nuform}. Denote
\begin{equation*}
    l(z):=(\widetilde{\psi}(x)+\nabla\widetilde{\psi}(x)\cdot(z-x))_+.
  \end{equation*}
  Then, for $x\in D$,
  \begin{equation*}
    \int_{|y|>c\widetilde{\psi}(x)} l^{\kappa_2}(x+y) \nu_{\kappa_1}(dy) \leq N \widetilde{\psi}^{-\kappa_1+\kappa_2}(x),
  \end{equation*}
  where $N$ depends only on $c,\kappa_1,\kappa_2$, $\Lambda$ and $D$.
\end{lem}

\begin{proof}
  Due to the definition of $\nu_{\kappa_1}$, we have
  \begin{align*}
    &\int_{|y|>c\widetilde{\psi}(x)} l^{\kappa_2}(x+y) \nu_{\kappa_1}(dy)
    \\
&= \int_{S^{d-1}} \int_{c\widetilde{\psi}(x)}^\infty \left[(\widetilde{\psi}(x)+\nabla\widetilde{\psi}(x)\cdot \theta r)_+ \right]^{\kappa_2} r^{-1-\kappa_1} \,dr \,\mu(d\theta)
=:\int_{S^{d-1}} I(\theta) \,\mu(d\theta).
  \end{align*}
  Therefore, it suffices to show
  \begin{equation} \label{eq8132057}
    I(\theta)\leq N \widetilde{\psi}^{-\kappa_1+\kappa_2}(x).
  \end{equation}
  First, if $\nabla\widetilde{\psi}(x)\cdot\theta=0$, then one can easily obtain \eqref{eq8132057}.
  Next, assume $\nabla\widetilde{\psi}(x)\cdot\theta>0$.
Note that if $\kappa_2\geq0$ and $s\geq 0$, then
\begin{align*}
  (1+cs)^{\kappa_2} \leq N(c,\kappa_2)(1 + s^{\kappa_2}),
\end{align*}
and if $\kappa_2<0$ and $s\geq0$, then
\begin{equation*}
  (1+cs)^{\kappa_2} \leq 1.
\end{equation*}
Thus, by the change of variables $\nabla\widetilde{\psi}(x)\cdot \theta r\to c\widetilde{\psi}(x) s$,
  \begin{align*}
    I(\theta) &= c^{-\kappa_1}\widetilde{\psi}^{-\kappa_1+\kappa_2}(x) \left(\nabla\widetilde{\psi}(x)\cdot \theta \right)^{\kappa_1} \int_{\nabla\widetilde{\psi}(x)\cdot \theta}^\infty (1+cs)^{\kappa_2} s^{-1-\kappa_1} \,ds
    \\
    &\leq N \widetilde{\psi}^{-\kappa_1+\kappa_2}(x) \left(\nabla\widetilde{\psi}(x)\cdot \theta \right)^{\kappa_1} \int_{\nabla\widetilde{\psi}(x)\cdot \theta}^\infty (s^{-1-\kappa_1} + s^{-1-\kappa_1+\kappa_2}) \,ds
    \\
    &\leq N \widetilde{\psi}^{-\kappa_1+\kappa_2}(x) \left( 1 + \left(\nabla\widetilde{\psi}(x)\cdot \theta \right)^{-1+\kappa_2} \right) \leq N \widetilde{\psi}^{-\kappa_1+\kappa_2}(x).
  \end{align*}
  Here, the last inequality follows from the fact $-1+\kappa_2>0$ and $\nabla\widetilde{\psi}$ is bounded.
  Lastly, it remains to prove \eqref{eq8132057} when $\nabla\widetilde{\psi}(x)\cdot\theta<0$. Similarly, by the change of variables $-\nabla\widetilde{\psi}(x)\cdot \theta r\to c\widetilde{\psi}(x) s$
  \begin{align*}
    I(\theta) &= c^{-\kappa_1}\widetilde{\psi}^{-\kappa_1+\kappa_2}(x) \left(-\nabla\widetilde{\psi}(x)\cdot \theta \right)^{\kappa_1} \int_{-\nabla\widetilde{\psi}(x)\cdot \theta}^\infty [(1-cs)_+]^{\kappa_2} s^{-1-\kappa_1} \,ds.
  \end{align*}
  Here,
  \begin{align*}
    &\int_{-\nabla\widetilde{\psi}(x)\cdot \theta}^\infty [(1-cs)_+]^{\kappa_2} s^{-1-\kappa_1} \,ds
    \\\
    &\leq 1_{-\nabla\widetilde{\psi}(x)\cdot \theta<1/2} \int_{-\nabla\widetilde{\psi}(x)\cdot \theta}^{1/2c} \cdots ds + \int_{1/2c}^{1/c} \cdots ds
    \\
    &\leq N \int_{-\nabla\widetilde{\psi}(x)\cdot \theta}^{\infty} s^{-1-\kappa_1} \,ds  + N \int_{1/2c}^{1/c} [(1-cs)_+]^{\kappa_2} \,ds
    \\
    &\leq N \left[ \left(-\nabla\widetilde{\psi}(x)\cdot \theta \right)^{-\kappa_1} + 1\right].
  \end{align*}
  Thus, we have \eqref{eq8132057}. The lemma is proved.
\end{proof}

The following is an extension of \cite[Proposition 2.3]{RSdom}.

\begin{lem} \label{lem_dist_dom}
Let $D$ be a bounded $C^{1,\tau}$ convex domain. Then, for any $-1+\alpha/2<\beta<\alpha/2$, there exists $\delta>0$ such that
\begin{equation*}
  L(\widetilde{\psi}^{\beta})(x) \leq -N d_x^{\beta-\alpha}, \quad 0<d_x<\delta,
\end{equation*}
where $N$ and $\delta$ depend only on $\alpha, \beta, \tau, \lambda, \Lambda, d$ and $D$.
\end{lem}

\begin{proof}
Since $\widetilde\psi$ is a regularized distance, one can take $\delta_1>0$ such that
\begin{align} \label{eq8031518}
  \inf_{0<d_z<\delta_1}|\nabla\widetilde{\psi}(z)|>0.
\end{align}
  Fix $x\in D$ such that $d_x<\delta_1$, and define
  \begin{equation*}
    l(z):=(\widetilde{\psi}(x)+\nabla\widetilde{\psi}(x)\cdot(z-x))_+.
  \end{equation*}
  Then, by Corollary \ref{cor_hs}, \eqref{nonde}, and \eqref{eq8031518}, there exists $N_0>0$ such that
  \begin{equation} \label{eq5261504}
    L(l^\beta)(x)\leq -N_0 l^{\beta-\alpha}(x) = -N_0 \widetilde{\psi}^{\beta-\alpha}(x).
  \end{equation}

  Now, we estimate $L(\widetilde{\psi}^\beta-l^\beta)(x)$.
Let
\begin{equation} \label{eq3041643}
  c:=\frac{1}{2\sup_{x\in \overline{D}}|\nabla\widetilde{\psi}(x)|}.
\end{equation}
  Since $\widetilde{\psi}^\beta(x)=l^\beta(x)$ and $\nu$ is symmetric,
  \begin{align} \label{eq5261639}
    |L(\widetilde{\psi}^\beta-l^\beta)(x)|&\leq\int_{\bR^d} |\widetilde{\psi}^\beta(x+y)-l^\beta(x+y)| \,\nu(dy) \nonumber
    \\
    &=\int_{|y|\leq c\widetilde{\psi}(x)} \cdots \nu(dy) + \int_{|y|> c\widetilde{\psi}(x)} \cdots \nu(dy) \nonumber
    \\
    &=:I_1(x)+I_2(x).
  \end{align}

  First, we estimate $I_1(x)$.
Note that due to \eqref{eq3041643}, $B_{c\widetilde{\psi}(x)}(x)\subset D$, and for $|y|\leq c\widetilde{\psi}(x)$,
\begin{equation*}
  \frac{1}{2}\widetilde{\psi}(x) \leq \widetilde{\psi}(x+y) \leq \frac{3}{2}\widetilde{\psi}(x).
\end{equation*}
Thus, for $|y|\leq c\widetilde{\psi}(x)$,
\begin{equation*}
  \widetilde{\psi}^{\beta-1}(x+y) \leq N \widetilde{\psi}^{\beta-1}(x), \quad l^{\beta-1}(x+y) \leq N \widetilde{\psi}^{\beta-1}(x).
\end{equation*}
Since $|D^2_x\widetilde{\psi}(x)|\leq N\widetilde{\psi}^{\tau-1}(x)$, for $|y|\leq c\widetilde{\psi}(x)$, by \eqref{eq5261440} with $\varepsilon=1$,
  \begin{align*}
    |\widetilde{\psi}^\beta(x+y)-l^\beta(x+y)| &\leq |\widetilde{\psi}(x+y)-l(x+y)||\widetilde{\psi}^{\beta-1}(x+y)+l^{\beta-1}(x+y)|
    \\
    &\leq N\widetilde{\psi}^{\beta-1}(x) |\widetilde{\psi}(x+y)-\widetilde{\psi}(x)-\nabla\widetilde{\psi}(x)\cdot y|
    \\
    &\leq N \widetilde{\psi}^{\tau+\beta-2}(x) |y|^2.
  \end{align*}
  Thus,
  \begin{equation} \label{eq5261651}
    I_1(x) \leq N\widetilde{\psi}^{\tau+\beta-\alpha}(x).
  \end{equation}
  Next, we consider $I_2(x)$. Since $\widetilde{\psi} \in C^{1,\tau}(\overline{D})$, we can consider $\widetilde{\psi}_0$, a $C^{1,\tau}(\bR^d)$ extension of $\widetilde{\psi}|_D$ satisfying $\widetilde{\psi}_0\leq0$ in $D^c$. Then, there exists $z\in B_{|y|}$ such that
  \begin{align} \label{eq5261713}
     |\widetilde{\psi}(x+y)-l(x+y)| &= | (\widetilde{\psi}_0(x+y))_+ - (\widetilde{\psi}(x)+\nabla\widetilde{\psi}(x)\cdot y)_+| \nonumber
     \\
     &\leq |\widetilde{\psi}_0(x+y)-\widetilde{\psi}(x)-\nabla\widetilde{\psi}(x)\cdot y| \nonumber
     \\
     &\leq |y\cdot(\nabla\widetilde{\psi}_0(x+z)-\nabla\widetilde{\psi}(x))| \nonumber
     \\
     &\leq N |y|^{1+\tau}.
   \end{align}
   Take $\varepsilon$ satisfying $0 < \varepsilon < \min\{1,\beta+1,\frac{\alpha-\beta}{\tau},\frac{\alpha}{1+\tau}\}$. By \eqref{eq5261440} and \eqref{eq5261713},
  \begin{align} \label{eq526525}
  &|\widetilde{\psi}^\beta(x+y)-l^\beta(x+y)| \nonumber
  \\
  &\leq |\widetilde{\psi}(x+y)-l(x+y)|^\varepsilon|\widetilde{\psi}^{\beta-\varepsilon}(x+y)+l^{\beta-\varepsilon}(x+y)| \nonumber
  \\
  &\leq N|y|^{\varepsilon+\varepsilon\tau}|\widetilde{\psi}^{\beta-\varepsilon}(x+y)+l^{\beta-\varepsilon}(x+y)|.
  \end{align}
  Since $c\widetilde{\psi}(x)>\kappa_0 d_x$ for some $\kappa_0>0$, by Corollary \ref{cor5261717} with $(\rho,\kappa_1,\kappa_2)=(\kappa_0 d_x,\alpha-\varepsilon-\varepsilon\tau,\beta-\varepsilon)$, we have
  \begin{align} \label{eq526526}
    \int_{|y|> c\widetilde{\psi}(x)} |y|^{\varepsilon+\varepsilon\tau} \widetilde{\psi}^{\beta-\varepsilon}(x+y) \,\nu(dy) &\leq N\int_{|y|> \kappa_0 d_x} |y|^{\varepsilon+\varepsilon\tau} d_{x+y}^{\beta-\varepsilon} \,\nu(dy) \nonumber
\\
&\leq N d_{x}^{\varepsilon\tau+\beta-\alpha} \leq N \widetilde{\psi}^{\varepsilon\tau+\beta-\alpha}(x).
  \end{align}
  Applying Lemma \ref{lem8132213} with $(\alpha-\varepsilon-\varepsilon\tau,\beta-\varepsilon)$ instead of $(\kappa_1,\kappa_2)$, we get
  \begin{align} \label{eq526527}
    \int_{|y|> c\widetilde{\psi}(x)} |y|^{\varepsilon+\varepsilon\tau} l^{\beta-\varepsilon}(x+y) \,\nu(dy) &\leq N \widetilde{\psi}^{\varepsilon\tau+\beta-\alpha}(x).
  \end{align}
  Thus, \eqref{eq526525}-\eqref{eq526527} lead to
  \begin{equation*}
    I_2(x)\leq N \widetilde{\psi}^{\varepsilon\tau+\beta-\alpha}(x).
  \end{equation*}
Therefore, this, \eqref{eq5261639}, and \eqref{eq5261651} yield
\begin{equation*}
  |L(\widetilde{\psi}^\beta-l^\beta)(x)| \leq N \widetilde{\psi}^{\varepsilon\tau+\beta-\alpha}(x).
\end{equation*}
Combining this with \eqref{eq5261504}, we get
\begin{equation*}
  L(\widetilde{\psi}^\beta)(x) \leq -N_0 \widetilde{\psi}^{\beta-\alpha}(x) + N \widetilde{\psi}^{\varepsilon\tau+\beta-\alpha}(x), \quad 0<d_x<\delta_1.
\end{equation*}
Thus, if $\delta<\delta_1$ is small enough, then we have the desired result. The lemma is proved.
\end{proof}

We also obtain the similar result for general $C^{1,\tau}$ open sets.

  \begin{lem} \label{lem_dist_open}
Let $D$ be a bounded $C^{1,\tau}$ open set. Then, for any $0<\beta<\alpha/2$, there exists $\delta>0$ such that
\begin{equation*}
  L(\widetilde{\psi}^{\beta})(x) \leq -N d_x^{\beta-\alpha}, \quad 0<d_x<\delta,
\end{equation*}
where $N$ depends only on $\alpha,\beta, \lambda,d,D$.
\end{lem}

\begin{proof}
  One can easily prove the lemma by following the proof of Lemma \ref{lem_dist_dom}.
  The main difference is that one needs to use Lemma \ref{lem5261746} instead of Corollary \ref{cor5261717} to obtain \eqref{eq526526}. The lemma is proved.
\end{proof}

The following lemma will be used to handle interior estimates of solutions.
\begin{lem} \label{lem_int}
Let $D$ be a bounded open set. Then, for $x\in D$,
\begin{equation*}
L1_{D}(x) \leq -N,
\end{equation*}
where $N$ depends only on $\alpha,D$ and $\lambda$.
\end{lem}

\begin{proof}
Since $D$ is bounded, one can find $R>0$ such that $D\subset B_R$. Then,
\begin{align*}
  L1_{D}(x)= -\int_{\bR^d \setminus\{0\}} 1_{D^c}(x+y) \,\nu(dy) \leq -\int_{B_{2R}^c} \,\nu(dy) \leq -N.
\end{align*}
The lemma is proved.
\end{proof}

\begin{remark} \label{rem8191725}
  Let $D$ be a convex domain (not necessarily bounded). Then, a similar result of Lemma \ref{lem_int} can be obtained; for $x\in D$,
  \begin{equation*}
  L1_{D}(x) \leq -Nd_x^{-\alpha},
\end{equation*}
where $N$ depends only on $\alpha$ and $\lambda$. Indeed, let $\theta_0\in S^{d-1}$ such that $x+d_x\theta_0\in \partial D$. Then, since $D$ is convex, for $r_\theta>0$ and $\theta\in S^{d-1}$ such that $x+r_\theta \theta \in \partial D$, we get $r_\theta(\theta\cdot\theta_0)\leq d_x$. Thus,
\begin{align*}
  L1_{D}(x) &= -\int_{\bR^d \setminus\{0\}} 1_{D^c}(x+y) \,\nu(dy) \leq -\int_{S^{d-1}} \int_{r_\theta}^\infty 1_{\theta\cdot\theta_0>0} r^{-1-\alpha} \,dr\mu(d\theta)
  \\
  &\leq -\int_{S^{d-1}} \int_{\frac{d_x}{\theta\cdot\theta_0}}^\infty 1_{\theta\cdot\theta_0>0} r^{-1-\alpha} \,dr\mu(d\theta)
  \\
  &\leq -\frac{d_x^{-\alpha}}{2\alpha}  \int_{S^{d-1}} |\theta\cdot\theta_0|^\alpha \,\mu(d\theta) \leq -N d_x^{-\alpha}.
\end{align*}
\end{remark}

\mysection{A priori estimates for solutions}

In this section, we obtain a priori estimates for solutions.

\subsection{Zeroth order estimates} \label{sec_zero}

We first prove the weighted Hardy-type inequality for $L$ on the half space.

\begin{lem} \label{lem_Hardy}
Let $1<p<\infty$, $-1+\alpha-\alpha p/2<c <p-1+\alpha-\alpha p/2$, and $u\in C_c^\infty(\bR_+^d)$. Then, under Assumption \ref{assum}, there exists $N=N(d,\alpha,c,p,\lambda)>0$ such that
\begin{equation*}
\int_{\bR_+^d} |u|^p x_1^{c-\alpha}\, dx \leq -N\int_{\bR_+^d} |u|^{p-2}uLu x_1^c\, dx.
\end{equation*}
\end{lem}

\begin{proof}
Here, we use the notation
\begin{equation*}
  f_y(x):=f(x+y),
\end{equation*}
and denote $h(x):=(x_1)_+$.
Then, by \eqref{op_rep},
\begin{align} \label{eq_3_2}
&-\int_{\bR_+^d} |u(x)|^{p-2}u(x)Lu(x) x_1^c\, dx \nonumber
\\
&=\frac{1}{2}\int_{\bR_+^d}\int_{\bR^d}  |u(x)|^{p-2}u(x) h^c(x) \Big(2u(x)-u_y(x)-u_{-y}(x)\Big) \,\nu(dy)\, dx.
\end{align}
Due to the range of $c$, we can take $\gamma\in\bR$ satisfying
\begin{equation} \label{eq_3_1}
-1+\frac{\alpha}{2}<\gamma<\frac{\alpha}{2}
\end{equation}
and
\begin{equation} \label{eq_3_10}
-1+\frac{\alpha}{2}<\gamma(p-1)+c<\frac{\alpha}{2}.
\end{equation}
Let $v(x):=u(x)h^{-\gamma}(x)$.
Then, for each $x\in \bR_+^d$ and $y\in \bR^d\setminus\{0\}$,
\begin{align} \label{eq_3_3}
  &|u(x)|^{p-2}u(x)h^c(x) \Big(2u(x)-u_y(x)-u_{-y}(x)\Big) \nonumber
  \\
  &= \frac{p-1}{p} v(x) |u(x)|^{p-2}u(x) h^c(x) \Big(2h^\gamma(x) - h_y^\gamma(x) - h_{-y}^\gamma(x)\Big) \nonumber
  \\
  &\quad+ \frac{p-1}{p} v(x) |u(x)|^{p-2}u(x) h^c(x) \Big(h_y^\gamma(x) + h_{-y}^\gamma(x)\Big) \nonumber
  \\
  &\quad+ \frac{1}{p} |v(x)|^p h^\gamma(x) \left(2h^{\gamma(p-1)+c}(x) - h_y^{\gamma(p-1)+c}(x) - h_{-y}^{\gamma(p-1)+c}(x)\right) \nonumber
  \\
  &\quad+ \frac{1}{p} |v(x)|^p h^\gamma(x) \left(h_y^{\gamma(p-1)+c}(x) + h_{-y}^{\gamma(p-1)+c}(x)\right) \nonumber
  \\
  &\quad- |u(x)|^{p-2}u(x)h^c(x) \Big(v_y(x)h_y^\gamma(x) + v_{-y}(x)h_{-y}^\gamma(x)\Big) \nonumber
  \\
  &=:\sum_{k=1}^5 I_k(x,y).
\end{align}
Recall that, in this paper, negative powers of $0$ is defined as $0$.
By \eqref{eq_3_1}, we can apply Corollary \ref{cor_hs} to get
\begin{equation*}
  Lh^\gamma(x) = -Nh^{\gamma-\alpha}(x),
\end{equation*}
with a constant $N>0$.
Thus,
\begin{align} \label{eq_3_4}
  \int_{\bR^d\setminus\{0\}} I_1(x,y)\,\nu(dy) &= -\frac{p-1}{p}v(x)|u(x)|^{p-2}u(x)h^c(x) Lh^\gamma(x) \nonumber
  \\
  &= N v(x)|u(x)|^{p-2}u(x)h^c(x) h^{\gamma-\alpha}(x) \nonumber
  \\
  &= N |u(x)|^p h^{c-\alpha}(x)=N|u(x)|^p x_1^{c-\alpha}.
\end{align}
Similarly, one can obtain
\begin{equation} \label{eq_3_5}
  \int_{\bR^d\setminus\{0\}} I_3(x,y) \,\nu(dy) = N |u(x)|^p x_1^{c-\alpha}.
\end{equation}
Next, we deal with $I_4$. Since $\nu$ is symmetric, for $\varepsilon>0$, the Fubini theorem yields
\begin{align*}
&\int_{\bR^d_+} \int_{|y|\geq\varepsilon} |v(x)|^p h^\gamma(x) h_y^{\gamma(p-1)+c}(x)\, \nu(dy) dx
\\
&= \int_{|y|\geq\varepsilon} \int_{\bR^d} |v(x)|^p h^\gamma(x) h_y^{\gamma(p-1)+c}(x)\, dx \nu(dy)
\\
&= \int_{|y|\geq\varepsilon} \int_{\bR^d} |v_{-y}(x)|^p h_{-y}^\gamma(x) h^{\gamma(p-1)+c}(x)\, dx \nu(dy)
\\
&= \int_{\bR_+^d} \int_{|y|\geq\varepsilon} |v_{-y}(x)|^p h_{-y}^\gamma(x) h^{\gamma(p-1)+c}(x)\, \nu(dy) dx.
\end{align*}
Thus,
\begin{align*}
  &\int_{\bR_+^d} \int_{|y|\geq\varepsilon} I_4(x,y)\, \nu(dy)dx
  \\
  &= \frac{1}{p} \int_{\bR_+^d} \int_{|y|\geq\varepsilon} |v_{y}(x)|^p h_{y}^\gamma(x) h^{\gamma(p-1)+c}(x)\, \nu(dy) dx
  \\
  &\quad+ \frac{1}{p} \int_{\bR_+^d} \int_{|y|\geq\varepsilon} |v_{-y}(x)|^p h_{-y}^\gamma(x) h^{\gamma(p-1)+c}(x) \, \nu(dy) dx.
\end{align*}
Using this, we have
\begin{align} \label{eq_3_22}
  &\lim_{\varepsilon\downarrow0} \int_{\bR_+^d} \int_{|y|\geq\varepsilon} \left( I_2(x,y) + I_4(x,y) + I_5(x,y) \right)\, \nu(dy) dx \nonumber
  \\
  &= \lim_{\varepsilon\downarrow0} \frac{1}{p} \int_{\bR_+^d} \int_{|y|\geq\varepsilon} \left[ |v_{y}(x)|^p - |v(x)|^p -p|v(x)|^{p-2}v(x) (v_y(x)-v(x)) \right] \nonumber
  \\
  &\qquad \qquad \qquad \qquad\qquad \qquad\qquad\qquad \qquad\quad \times h_y^\gamma(x) h^{\gamma(p-1)+c}(x) \, \nu(dy) dx \nonumber
  \\
  &\quad + \lim_{\varepsilon\downarrow0} \frac{1}{p} \int_{\bR_+^d} \int_{|y|\geq\varepsilon} \left[ |v_{-y}(x)|^p - |v(x)|^p -p|v(x)|^{p-2}v(x) (v_{-y}(x)-v(x)) \right] \nonumber
  \\
  &\qquad \qquad \qquad \qquad\qquad \qquad\qquad\qquad \qquad\quad \times h_{-y}^\gamma(x) h^{\gamma(p-1)+c}(x) \, \nu(dy) dx.
\end{align}
Due the convexity of function $a\to |a|^p$, we have
\begin{equation*}
  |b|^p-|a|^p-p|a|^{p-2}a(b-a) \geq 0, \quad a,b\in \bR.
\end{equation*}
Thus, we conclude that
\begin{equation*}
  \lim_{\varepsilon\downarrow0} \int_{\bR_+^d} \int_{|y|\geq\varepsilon} \left( I_2(x,y) + I_4(x,y) + I_5(x,y) \right)\, \nu(dy) dx \geq0.
\end{equation*}
Combining this with \eqref{eq_3_2} and \eqref{eq_3_3}-\eqref{eq_3_5}, we have the desired result. The lemma is proved.
\end{proof}

Now we deal with $C^{1,\tau}$ open sets.
\begin{lem} \label{lem_6151638}
  Let $D$ be a bounded $C^{1,\tau}$ convex domain, $1<p<\infty$, $-1+\alpha-\alpha p/2<c <p-1+\alpha-\alpha p/2$, and $u\in C_c^\infty(D)$. Then, under Assumption \ref{assum}, for a regularized distance $\widetilde{\psi}$ satisfying \eqref{eq5282250}, there exist constants $N_1$ and $N_2$ depending only on $d,\alpha,c,\tau,p,\lambda,\Lambda$ and $D$ such that
\begin{align} \label{eq_Hardy_dom}
\int_{D} |u|^p \widetilde{\psi}^{c-\alpha}\, dx &\leq  -N_1\int_{D} |u|^{p-2}uLu \widetilde{\psi}^c\, dx -N_2\int_{D} |u|^{p-2}uLu\, dx.
\end{align}

Moreover, if $D$ is a bounded $C^{1,\tau}$ open set, then the claim still holds provided that $\alpha/2-\alpha p/2<c<\alpha /2$.
\end{lem}

\begin{proof}
We first prove
\begin{align} \label{eq6151509}
\int_{D} |u|^p \widetilde{\psi}^{c-\alpha}\, dx &\leq  -N_1\int_{D} |u|^{p-2}uLu \widetilde{\psi}^c\, dx + N \int_{D_\delta} |u|^p \,dx,
\end{align}
where $D_\delta:=\{x\in D: d_x\geq \delta\}$.

 We repeat the proof of Lemma \ref{lem_Hardy} by substituting $\widetilde{\psi}(x)$ for $h=(x_1)_+$. Then, as in \eqref{eq_3_3} and \eqref{eq_3_22},
  \begin{align} \label{eq_3_19}
     -\int_{D} |u|^{p-2}uLu \widetilde{\psi}^c \,dx &\geq -\frac{p-1}{p} \int_{D} |u|^p \widetilde{\psi}^{c-\gamma} L\widetilde{\psi}^\gamma \,dx \nonumber
     \\
     &\quad- \frac{1}{p} \int_{D} |u|^p \widetilde{\psi}^{-\gamma(p-1)} L\widetilde{\psi}^{\gamma(p-1)+c} \,dx,
   \end{align}
   where $\gamma$ satisfies \eqref{eq_3_1} and \eqref{eq_3_10}. By Lemma \ref{lem_dist_dom}, there exists $\delta>0$ such that
   \begin{align} \label{eq_3_20}
   &\int_{E_\delta} |u|^p \widetilde{\psi}^{c-\alpha}\,dx \leq-N\int_{E_\delta} |u|^p \widetilde{\psi}^{c-\gamma} L\widetilde{\psi}^\gamma \,dx   - N \int_{E_\delta} |u|^p \widetilde{\psi}^{-\gamma(p-1)} L\widetilde{\psi}^{\gamma(p-1)+c} \,dx,
   \end{align}
   where $E_\delta:=\{x\in D: 0<d_x<\delta\}$. Since $\widetilde{\psi}$ is a regularized distance, for $\upsilon>-1$, $L\psi^\upsilon$ is bounded on $D_\delta:=\{x\in D: d_x\geq \delta\}$. Thus, we have
   \begin{align} \label{eq6151420}
     &\int_{D_\delta} |u|^p \widetilde{\psi}^{c-\gamma} |L\widetilde{\psi}^\gamma(x)| \,dx +\int_{D_\delta} |u|^p \widetilde{\psi}^{-\gamma(p-1)} |L\widetilde{\psi}^{\gamma(p-1)+c}| \,dx \leq N \int_{D_\delta} |u|^p \,dx.
   \end{align}
   Therefore, \eqref{eq6151509} follows from \eqref{eq_3_19}-\eqref{eq6151420}.

   Now we consider interior estimates for $u$.
   We again repeat the proof of Lemma \ref{lem_Hardy} with $h(x)=1_{D}(x)$. Then, by Lemma \ref{lem_int}, we have
   \begin{align*}
     \int_{D} |u(x)|^p \,dx \leq - N\int_{D} |u(x)|^p  L1_D(x) \,dx \leq -N_2\int_{D} |u(x)|^{p-2}u(x)Lu(x) \,dx.
   \end{align*}
   This together with \eqref{eq6151509} yields \eqref{eq_Hardy_dom}.

   Lastly, we deal with the case when $D$ is a bounded $C^{1,\tau}$ open set.
   Due to the range of $c$, we can take $\gamma\in\bR$ satisfying
\begin{equation*}
0<\gamma<\frac{\alpha}{2}
\end{equation*}
and
\begin{equation*}
0<\gamma(p-1)+c<\frac{\alpha}{2}.
\end{equation*}
Then, by repeating the above proof with Lemma \ref{lem_dist_open} in place of Lemma \ref{lem_dist_dom}, we have the desired result. The lemma is proved.
\end{proof}

\begin{lem} \label{lem6151427}
Let $1<p<\infty$, $0<T<\infty$, $d-1<\theta<d-1+p$ and $u\in C_c^\infty([0,T]\times\bR_+^d)$.
Then, under Assumption \ref{assum_t}, there exists $N=N(d,p,\alpha,\theta,\lambda)$ such that
\begin{align} \label{eq_pp}
  &\int_0^T \int_{\bR_+^d} |u|^p x_1^{\theta-d-\alpha p/2} \,dxdt + \int_{\bR_+^d} |u(T)|^p x_1^{\theta-d-\alpha p/2+\alpha} \,dx \nonumber
  \\
  &\leq N \int_0^T \int_{\bR_+^d} |f|^p x_1^{\theta-d+\alpha p/2} \,dxdt + N \int_{\bR_+^d} |u(0)|^p x_1^{\theta-d-\alpha p/2+\alpha} \,dx,
\end{align}
where $u(t):=u(t,\cdot)$ and
\begin{equation} \label{eq6142053}
   f:=\partial_tu-L_tu.
\end{equation}
\end{lem}

\begin{proof}
We first multiply both sides of \eqref{eq6142053} by $|u|^{p-2}ux_1^c$ where $c:=\theta-d-\alpha p/2+\alpha$. Then,
\begin{equation} \label{eq6151600}
\begin{aligned}
&\frac{1}{p}\int_0^T \int_{\bR_+^d} \partial_t (|u|^p) x_1^c \,dxdt - \int_0^T \int_{\bR_+^d} L_t u |u|^{p-2}ux_1^c \,dxdt \\
&= \int_0^T \int_{\bR_+^d} f |u|^{p-2}ux_1^c \,dxdt.
\end{aligned}
\end{equation}
Notice that, by the fundamental theorem of calculus,
\begin{align*}
  \int_0^T \int_{\bR_+^d} \partial_t (|u|^p) x_1^c \,dxdt &= \int_{\bR_+^d} |u(T)|^p x_1^c \,dx - \int_{\bR_+^d} |u(0)|^p x_1^c \,dx.
\end{align*}
Since $\nu_t$ satisfies Assumption \ref{assum} for each $t$, we can apply Lemma \ref{lem_Hardy} and H\"older's inequality to get
\begin{align*}
  &\int_0^T \int_{\bR_+^d} |u|^p x_1^{c-\alpha} \,dxdt + \int_{\bR_+^d} |u(T)|^p x_1^{\theta-d-\alpha p/2+\alpha} \,dx
  \\
  &\leq N\int_{\bR_+^d} |u(0)|^p x_1^c \,dx + N\int_0^T \int_{\bR_+^d} f |u|^{p-2}ux_1^c \,dxdt
  \\
  &\leq N\int_{\bR_+^d} |u(0)|^p x_1^c \,dx
  \\
  &\quad+ N\left( \int_0^T \int_{\bR_+^d} |f|^p x_1^{c+\alpha p-\alpha} \,dxdt \right)^{1/p}  \left( \int_0^T \int_{\bR_+^d} |u|^p x_1^{c-\alpha} \,dxdt \right)^{(p-1)/p}.
\end{align*}
  This gives \eqref{eq_pp}. The lemma is proved.
\end{proof}

Now we consider bounded $C^{1,\tau}$ convex domains.

\begin{lem} \label{lem_Hardy_dom}
Let $D$ be a bounded $C^{1,\tau}$ convex domain, $1<p<\infty$, $0<T<\infty$, $d-1 < \theta \leq d-\alpha+\alpha p/2$ and $u\in C_c^\infty([0,T]\times D)$. Then, under Assumption \ref{assum_t}, there exists $N=N(\theta,d,\alpha,p,\tau,\lambda,\Lambda,D)>0$ such that
\begin{align*}
  &\int_0^T \int_{D} |u|^p d_x^{\theta-d-\alpha p/2} \,dxdt + \int_{D} |u(T)|^p d_x^{\theta-d-\alpha p/2+\alpha} \,dx
  \\
  &\leq N \int_0^T \int_{D} |f|^p d_x^{\theta-d+\alpha p/2} \,dxdt + N \int_{D} |u(0)|^p d_x^{\theta-d-\alpha p/2+\alpha} \,dx,
\end{align*}
where $u(t):=u(t,\cdot)$ and $f:=\partial_t u-Lu$.

Moreover, if $D$ is a bounded $C^{1,\tau}$ open set, then the claim still holds provided that $d-\alpha/2<\theta\leq d - \alpha +\alpha p /2$.
\end{lem}

\begin{proof}
  Let $\widetilde{\psi}$ be a regularized distance satisfying \eqref{eq5282250}, and $N_1$ and $N_2$ be the constants taken from Lemma \ref{lem_6151638} with $c:=\theta-d-\alpha p/2+\alpha$. Here, the possible range of $\theta$ is determined according to the conditions of $D$. As in \eqref{eq6151600}, by multiplying both sides of $f=\partial_t u-Lu$ by $|u|^{p-2}u (N_1\widetilde{\psi}^c+N_2)$,
\begin{align*}
  &\frac{1}{p}\int_0^T \int_{D} \partial_t (|u|^p) (N_1\widetilde{\psi}^c+N_2) \,dxdt - \int_0^T \int_{D} L u |u|^{p-2}u (N_1\widetilde{\psi}^c+N_2) \,dxdt
  \\
  &= \int_0^T \int_{D} f |u|^{p-2}u(N_1\widetilde{\psi}^c+N_2) \,dxdt.
\end{align*}
Then, by following the proof of Lemma \ref{lem6151427} with Lemma \ref{lem_6151638} instead of Lemma \ref{lem_Hardy}, we get
\begin{align*}
  &\int_0^T \int_{D} |u|^p \widetilde{\psi}^{c-\alpha} \,dxdt + \int_{D} |u(T)|^p \widetilde{\psi}^c \,dx
  \\
  &\leq N \int_{D} |u(0)|^p (1+\widetilde{\psi}^c) \,dx \nonumber
  \\
  &\quad+ N \int_0^T \int_{D} |f| |u|^{p-1}(1+\widetilde{\psi}^c) \,dxdt
  \\
  &\leq  N \int_{D} |u(0)|^p \widetilde{\psi}^c \,dx
  \\
  &\quad+ N\left( \int_0^T \int_{D} |f|^p \widetilde{\psi}^{c+\alpha p-\alpha} \,dxdt \right)^{1/p}  \left( \int_0^T \int_{D} |u|^p \widetilde{\psi}^{c-\alpha} \,dxdt \right)^{(p-1)/p}.
\end{align*}
Here, for the last inequality, we used the fact that $D$ is bounded and $c\leq0$. Thus, the desired result follows from  $N^{-1}\widetilde{\psi}(x) \leq d_x \leq N\widetilde{\psi}(x)$. The lemma is proved.
\end{proof}

\begin{remark}
In the proofs of Lemmas \ref{lem6151427} and \ref{lem_Hardy_dom}, the condition \eqref{ineq9040040} is not necessary. In other words, those lemmas can be proved for $\nu_t$ which satisfies Assumption \ref{assum} for each $t$.
\end{remark}

\begin{remark}
When $D$ is a convex domain such that $\sup_{x\in D} d_x<\infty$, the same result of Lemma \ref{lem_Hardy_dom} can be obtained by following the proof of
Lemmas \ref{lem_6151638} and \ref{lem_Hardy_dom} with Remark \ref{rem8191725} in place of Lemma \ref{lem_int}.
\end{remark}

\subsection{Higher order estimates} \label{sec_high}

In this subsection, we obtain higher order regularity of solutions. To prove this, we extend the ideas of \cite{Dirichlet}, which treats the fractional Laplacian $\Delta^{\alpha/2}$.

Throughout this subsection, we fix a collection of functions $\{\zeta_n : n\in \bZ\}$ satisfying \eqref{zeta_1}-\eqref{zeta_3} with $(c_1,c_2)=(1,e^2)$. We also take $\{\eta_n : n\in \bZ\}$ satisfying $\zeta_n \eta_n=\zeta_n$ and \eqref{zeta_1}-\eqref{zeta_3} with $(c_1,c_2)=(e^{-2},e^4)$.

\begin{lem} \label{lem_peturb_1}
Let $D$ be an open set with nonempty boundary, $\gamma\in\bR$, $1<p<\infty$, and $u\in C_c^\infty(D)$. Then, under Assumption \ref{assum}, there exists a constant $N=N(d,\alpha,\Lambda,D)$ such that for any $n\in\bZ$,
\begin{align*}
    &\left\|L\Big((u\zeta_{-n}\eta_{-n})(e^n\cdot)\Big)-
    \zeta_{-n}(e^n\cdot)L\Big((u\eta_{-n})(e^n\cdot) \Big)\right\|_{H_p^\gamma}
    \\
    &\leq N\left\|\Delta^{\alpha/4}\Big((u\eta_{-n})(e^n\cdot)\Big) \right\|_{H_p^\gamma}+N\|u(e^n\cdot)\eta_{-n}(e^n\cdot)\|_{H_p^\gamma}.
\end{align*}
\end{lem}

\begin{proof}
Since $\nu$ is symmetric, by \eqref{op},
    \begin{align} \label{eq_3_6}
    &L\big((u\zeta_{-n}\eta_{-n})(e^n\cdot)\big)(x)-\zeta_{-n}(e^nx) L\big((u\eta_{-n})(e^n\cdot)\big)(x) \nonumber
    \\
    &-u(e^n x)\eta_{-n}(e^nx)L\zeta_{-n}(e^n\cdot)(x) = \int_{\bR^d}F_n(x,y)\,\nu(dy),
    \end{align}
where
$$
F_n(x,y):=[(u\eta_{-n})(e^n(x+y))-(u\eta_{-n})(e^nx)][\zeta_{-n}(e^n(x+y))-\zeta_{-n}(e^nx)].
$$
Due to \eqref{zeta_2}, we have
\begin{align*}
    \left|D^m_x\left(\zeta_{-n}(e^n(x+y))-\zeta_{-n}(e^nx)\right)\right| \leq N (1\wedge|y|),
\end{align*}
and thus $\zeta_{-n}(e^n(x+y))-\zeta_{-n}(e^nx)$ becomes a pointwise multiplier in $H_p^\gamma$ (see e.g. \cite[Lemma 5.2]{kry99analytic}).
Hence, we can apply \cite[Lemma 2.1]{Z13} to get
\begin{align*}
&\|F_n(\cdot,y)\|_{H_p^\gamma} \leq N (1\wedge|y|) \|(u\eta_{-n})(e^n(\cdot+y))-(u\eta_{-n})(e^n\cdot)\|_{H_p^\gamma}
\\
&\leq N\left(\|u(e^n\cdot)\eta_{-n}(e^n\cdot)\|_{H_p^{\gamma}} \wedge |y|^{\alpha/2+1}\|\Delta^{\alpha/4}\big(u(e^n\cdot)\eta_{-n}(e^n\cdot)\big)\|_{H_p^{\gamma}}\right).
\end{align*}
Thus, by Minkowski's inequality,
\begin{align} \label{eq_3_7}
    \left\| \int_{\bR^d}F_n(\cdot,y) \,\nu(dy) \right\|_{H_p^\gamma} &\leq N \|\Delta^{\alpha/4}\big(u(e^n\cdot)\eta_{-n}(e^n\cdot)\big)\|_{H_p^\gamma} \int_{|y|\leq1}|y|^{\alpha/2+1}\,\nu(dy) \nonumber
    \\
    &\quad+ N \|u(e^n\cdot)\eta_{-n}(e^n\cdot)\|_{H_p^\gamma} \int_{|y|>1}\,\nu(dy) \nonumber
    \\
    &\leq N\|\Delta^{\alpha/4}\big((u\eta_{-n})(e^n\cdot)\big)\|_{H_p^\gamma}+N\|u(e^n\cdot)\eta_{-n}(e^n\cdot)\|_{H_p^\gamma}.
\end{align}
On the other hand, by \eqref{zeta_2}, one can find that $|D_x^m L(\zeta_{-n}(e^n\cdot))|$ is uniformly bounded with respect to $n$. Therefore, \cite[Lemma 5.2]{kry99analytic} yields
\begin{equation*}
    \|u(e^n \cdot)\eta_{-n}(e^n\cdot)L(\zeta_{-n}(e^n\cdot))\|_{H_p^\gamma}\leq N\|u(e^n \cdot)\eta_{-n}(e^n\cdot)\|_{H_p^\gamma}.
\end{equation*}
This, \eqref{eq_3_6} and \eqref{eq_3_7} lead to the desired result. The lemma is proved.
\end{proof}

The following lemma is an extension of \cite[Lemma 4.2]{Dirichlet}.
\begin{lem} \label{lem_peturb_2_dom}
Let $D$ be a convex domain with nonempty boundary, $1<p<\infty$, $u\in C_c^\infty(D)$ and $d-1-\alpha p/2<\theta<d+p-1+\alpha p/2$.
Then, under Assumption \ref{assum},
\begin{align*}
&\sum_{n\in\bZ} e^{n(\theta-\alpha p/2)} \left\|\zeta_{-n}(e^n\cdot)L\Big( [1-\eta_{-n}(e^n\cdot)]u(e^n\cdot) \Big) \right\|_{L_p}^p\leq N \|u\|_{L_{p,\theta-\alpha p/2}(D)}^p.
\end{align*}
where $N$ depends only on $d,p,\alpha,\theta,D$ and $\Lambda$.

Moreover, if $D$ is an open set with nonempty boundary, then the claim still holds provided that $d-\alpha p/2<\theta<d+\alpha p/2$.
\end{lem}

\begin{proof}
\textbf{1.} We first treat the case $D$ is a convex domain with nonempty boundary.
Notice that
\begin{equation*}
\zeta_{-n}(e^n x)(1-\eta_{-n}(e^n(x+y)))=0 \quad \text{if} \quad  |y|< \delta_0,
\end{equation*}
where $\delta_0:=1-e^{-1}$. Hence, due to the \eqref{op} (recall that $\nu$ is symmetric),
\begin{align*}
&\zeta_{-n}(e^nx)L\Big( [1-\eta_{-n}(e^n\cdot)]u(e^n\cdot)\Big)(x)
\\
&=\zeta_{-n}(e^nx)\int_{|y|\geq \delta_0} u(e^n(x+y)) [1-\eta_{-n}(e^n(x+y))]\,\nu(dy)
\\
&\leq N\zeta_{-n}(e^nx)\int_{|y|\geq \delta_0} |u(e^n(x+y))| \,\nu(dy) =:F_n(x).
\end{align*}
Thus, we only need to show
\begin{align} \label{eq_3_15}
\sum_{n\in\bZ} e^{n(\theta-\alpha p/2)} \|F_n\|_{L_p}^p \leq N \|u\|_{L_{p,\theta-\alpha p/2}(D)}^p.
\end{align}

Let $\beta\in(0,\alpha)$ be given. Since $d-1-\alpha p/2<\theta<d+p-1+\alpha p/2$, we can take $\gamma\in\bR$ such that
\begin{equation} \label{eq8141333}
  1-p<\gamma p<\alpha p-\beta p,
\end{equation}
and
\begin{equation} \label{eq8141334}
  -1<\theta-d-\alpha p/2 +\gamma p + \beta p < \beta p.
\end{equation}
By H\"older's inequality,
\begin{align} \label{eq_3_14}
F_n(x)&\leq \zeta_{-n}(e^nx)\left(\int_{S^{d-1}} \int_{r\geq \delta_0} d_{e^n(x+r\theta)}^{-\gamma p} |u(e^n(x+r\theta))|^p r^{-1-\beta p} \,dr \mu(d\theta) \right)^{1/p} \nonumber
\\
& \qquad\qquad\times\left(\int_{S^{d-1}} \int_{r\geq \delta_0} d_{e^n(x+r\theta)}^{\gamma p'} r^{-1-(\alpha-\beta)p'} \,dr \mu(d\theta) \right)^{1/p'}.
\end{align}
where $p':=p/(p-1)$ and $\mu$ is the spherical part of $\nu$.
By the change of variables, \eqref{eq8141333}, and Corollary \ref{cor5261717} with $(\rho,\kappa_1,\kappa_2)=(e^n\delta_0,(\alpha-\beta)p',\gamma p')$, for $x\in \text{supp}(\zeta_{-n}(e^n\cdot))$,
\begin{align*}
&\int_{S^{d-1}} \int_{r\geq \delta_0} d_{e^n(x+r\theta)}^{\gamma p'} r^{-1-(\alpha-\beta)p'} \,dr \mu(d\theta)
\\
&=e^{n(\alpha-\beta)p'}\int_{S^{d-1}} \int_{r\geq e^n\delta_0} d_{e^nx+r\theta}^{\gamma p'} r^{-1-(\alpha-\beta)p'} \,dr \mu(d\theta)
  \\
  &\leq N d_{e^nx}^{\gamma p'-(\alpha-\beta)p'} e^{n(\alpha-\beta)p'} \leq N e^{n\gamma p'}.
\end{align*}
This and \eqref{eq_3_14} yield
\begin{align*}
F_n(x) \leq Ne^{n\gamma}\zeta_{-n}(e^nx)\left(\int_{|y|\geq \delta_0} d_{e^n(x+y)}^{-\gamma p}|u(e^n(x+y))|^p \,\nu_{\beta p}(dy) \right)^{1/p},
\end{align*}
where $\nu_{\beta p}$ is a measure taking the form \eqref{nuform}.
Then, by the Fubini theorem and the change of variables $(e^nx,e^ny)\to(x,y)$,
\begin{align}  \label{eq_3_16}
&\sum_{n\in\bZ} e^{n(\theta-\alpha p/2)} \|F_n\|_{L_p}^p \nonumber
\\
&\leq N\sum_{n\in\bZ} e^{n(\theta-\alpha p/2+\gamma p)}  \int_{|y|\geq \delta_0} \int_{\bR^d}  |\zeta_{-n}(e^n x)|^p d_{e^nx+y}^{-\gamma p}|u(e^n(x+y))|^p \,dx\nu_{\beta p}(dy) \nonumber
\\
&= N(d) \sum_{n\in\bZ} e^{n(\theta-d-\alpha p/2+\gamma p+\beta p)} \int_{|y|\geq e^n\delta_0} \int_{D}  |\zeta_{-n}(x)|^p d_{x+y}^{-\gamma p}|u(x+y)|^p \,dx\nu_{\beta p}(dy) \nonumber
\\
&=: N(d) \int_{D} H(x) d_x^{-\gamma p} |u(x)|^p \,dx,
\end{align}
where
\begin{equation*}
  H(x):= \sum_{n\in\bZ} e^{n(\theta-d-\alpha p/2+\gamma p+\beta p)}\int_{|y|\geq e^n\delta_0}|\zeta_{-n}(x-y)|^p \, \nu_{\beta p}(dy).
\end{equation*}
Here, for the last equality in \eqref{eq_3_16}, we used the change of variables $x\to x-y$.

Now we estimate $H(x)$.
For  fixed $x\in D$, there exists $n_0=n_0(x)\in\bZ$ such that
$$
e^{n_0+3} \leq d_x < e^{n_0+4}.
$$
If $n\leq n_0$ and $x-y\in \text{supp}(\zeta_{-n})$, then  $e^n<d_{x-y}<e^{n+2}\leq e^{n_0+2} \leq e^{-1}d_x$, and consequently $|y|\geq d_x - d_{x-y} \geq N e^{n_0}$.
Using this relation, \eqref{eq8141334}, and Corollary \ref{cor5261717} with $(\rho,\kappa_1,\kappa_2)=(Ne^{n_0},\beta p,\theta-d-\alpha p/2 +\gamma p+\beta p)$ (recall that $\nu$ is symmetric),
\begin{align} \label{eq_3_17}
&\sum_{n\leq n_0} e^{n(\theta-d-\alpha p/2+\gamma p+\beta p)}\int_{|y|\geq e^n\delta_0}|\zeta_{-n}(x-y)|^p \nu_{\beta p}(dy) \nonumber
\\
 &\leq N\int_{|y|\geq Ne^{n_0}}  \sum_{n\leq n_0} |\zeta_{-n}(x-y)|^p d_{x-y}^{\theta-d-\alpha p/2+\gamma p+\beta p}\nu_{\beta p}(dy) \nonumber
\\
&\leq N\int_{|y|\geq Ne^{n_0}} d_{x-y}^{\theta-d-\alpha p/2+\gamma p+\beta p}\nu_{\beta p}(dy) \nonumber
\\
&\leq Nd_x^{\theta-d-\alpha p/2+\gamma p}.
\end{align}

Now we consider the summation for $n>n_0$. Due to $\theta-d-\alpha p/2+\gamma p<0$,
    \begin{align*}
    &\sum_{n>n_0}e^{n(\theta-d-\alpha p/2+\gamma p+\beta p)}\int_{|y|\geq \delta_0e^n} |\zeta_{-n}(x-y)|^p \nu_{\beta p}(dy)
    \\
    &\leq N \sum_{n>n_0}e^{n(\theta-d-\alpha p/2+\gamma p+\beta p)}\int_{|y|\geq \delta_0e^n} \nu_{\beta p}(dy)
    \\
    &\leq N\sum_{n>n_0}e^{n(\theta-d-\alpha p/2+\gamma p)} = Ne^{n_0(\theta-d-\alpha p/2+\gamma p)} \leq N d_x^{\theta-d-\alpha p/2+\gamma p}.
\end{align*}
This together with \eqref{eq_3_17} leads to
\begin{equation*}
  H(x)\leq d_x^{\theta-d-\alpha p/2+\gamma p}.
\end{equation*}
Thus, by \eqref{eq_3_16}, we obtain \eqref{eq_3_15}.

\textbf{2.} Now we deal with open sets with nonempty boundary.
This case can be obtained by repeating the above argument. More specifically, for given $\beta\in(0,\alpha)$, take $\gamma\in\bR$ such that
\begin{equation*}
  0\leq\gamma p<\alpha p-\beta p,
\end{equation*}
and
\begin{equation*}
  0\leq\theta-d-\alpha p/2 +\gamma p + \beta p < \beta p,
\end{equation*}
instead of \eqref{eq8141333} and \eqref{eq8141334}, respectively. For instance, in this case, we can choose $\gamma=0$.
 Then, proceed the proof with Lemma \ref{lem5261746} instead of Corollary \ref{cor5261717}.
 The lemma is proved.
\end{proof}

\begin{lem} \label{lem6091055}
Let $D$ be a convex domain with nonempty boundary, $\gamma\leq0$, $1<p<\infty$, $d-1-\alpha p/2<\theta<d-1+p+\alpha p/2$, and $u \in  C^{\infty}_c(D)$. Then, under Assumption \ref{assum}, there exists a constant $N=N(d,p,\alpha,\theta,D,\Lambda)$ such that
\begin{align} \label{lem_peturb}
&\sum_{n\in\bZ} e^{n(\theta-\alpha p/2)} \left\|L\Big(u(e^n\cdot)\zeta_{-n}(e^n\cdot)\Big)-\zeta_{-n}(e^n\cdot)L(u(e^n\cdot)) \right\|_{H_p^\gamma}^p \nonumber
\\
&\leq N \|\psi^{-\alpha/2}u\|_{H_{p,\theta}^{0\vee(\gamma+\alpha/2)}(D)}^p.
\end{align}

Moreover, if $D$ is an open set with nonempty boundary, then the claim still holds provided that $d-\alpha p/2<\theta<d+\alpha p/2$.
\end{lem}

\begin{proof}
Since $\eta_{-n}\zeta_{-n}=\zeta_{-n}$, by the triangle inequality and the relation $L_p\subset H_p^\gamma$,
\begin{align*}
& \left\|L\Big(u(e^n\cdot)\zeta_{-n}(e^n\cdot)\Big)-\zeta_{-n}(e^n\cdot)L(u(e^n\cdot)) \right\|_{H_p^\gamma}
\\
&\leq  \left\|L\Big((u\zeta_{-n}\eta_{-n})(e^n\cdot)\Big)-
    \zeta_{-n}(e^n\cdot)L\Big((u\eta_{-n})(e^n\cdot) \Big)\right\|_{H_p^\gamma}
    \\
    &\quad + \left\|\zeta_{-n}(e^n\cdot)L\Big( [1-\eta_{-n}(e^n\cdot)]u(e^n\cdot) \Big) \right\|_{L_p}.
         \end{align*}
Thus, by Lemmas \ref{lem_peturb_1} and \ref{lem_peturb_2_dom}, \eqref{defSob}, and the relation
\begin{equation*}
 \|v\|_{H_p^\gamma}+\|\Delta^{\alpha/4}v\|_{H_p^\gamma} \leq N\|v\|_{H^{\gamma+\alpha/2}_p},
\end{equation*}
we have \eqref{lem_peturb}.
The lemma is proved.
\end{proof}

 For a distribution $u$ on an open set $U\subset \bR^d$, the action of $u$ on $\phi\in C_c^\infty(U)$ is denoted by
\begin{equation} \label{eq8201210}
(u,\phi)_U:=u(\phi).
\end{equation}
Due to Lemma \ref{lem_prop} $(iii)$, for $U=D$ and $u\in H_{p,\theta}^\gamma(D)$, \eqref{eq8201210}, defined on $C_c^\infty(D)$, can be extended by continuity to $H_{p',\theta'}^{-\gamma}(D)$.

The following lemma shows that the boundedness of $L$ from $\psi^{\alpha/2}H_{p,\theta}^\alpha(D)$ to $\psi^{-\alpha/2}L_{p,\theta}(D)$.

\begin{lemma} \label{lem6101710}
Let $1<p<\infty$ and Assumption \ref{assum} hold.

(i) Let $D$ be a convex domain with nonempty boundary, and $d-1-\alpha p/2<\theta<d-1+p+\alpha p/2$. Then, for any $u\in C_c^\infty(D)$, we have $Lu \in \psi^{-\alpha/2}L_{p,\theta}(D)$ and
\begin{equation*}
\|\psi^{\alpha/2} Lu\|_{L_{p,\theta}(D)}\leq N \|\psi^{-\alpha/2}u\|_{H^{\alpha}_{p,\theta}(D)},
\end{equation*}
where $N=N(d,p,\alpha,\theta,D,\lambda,\Lambda)$.

(ii) Under the same conditions in (i), for $u\in \psi^{\alpha/2}H_{p,\theta}^{\alpha}(D)$, $Lu$ defined as
\begin{equation*}
  (Lu,\phi)_D=(u,L\phi)_D, \quad \phi\in C_c^\infty(D),
\end{equation*}
is well defined and belongs to $\psi^{-\alpha/2}L_{p,\theta}(D)$.

Moreover, if $D$ is an open set with nonempty boundary, then all the claims above still hold provided that $d-\alpha p/2<\theta<d+\alpha p/2$.
\end{lemma}

\begin{proof}
First, we prove $(i)$. By \cite[Proposition $1$]{MP2017}, we have
  \begin{equation} \label{eq6211501}
    N^{-1} \|\Delta^{\alpha/2}v\|_{L_p} \leq \|Lv\|_{L_p} \leq N\|\Delta^{\alpha/2}v\|_{L_p}.
  \end{equation}

  Then, the claim of $(i)$ easily follows from Lemma \ref{lem6091055} with $\gamma=0$ and the relations \eqref{eq6211501} and $Lv(e^nx)=e^{-n\alpha}L(v(e^n\cdot))(x)$.

  Now we consider $(ii)$. By $(i)$ and Lemma \ref{lem_prop} $(iii)$, for $\phi\in C_c^\infty(D)$, $L\phi$ is in the dual space of $\psi^{\alpha/2}H_{p,\theta}^{\alpha}(D)$. This actually proves $(ii)$.
  The lemma is proved.
\end{proof}

Now we are ready to prove higher order regularity of solutions.

\begin{lem} \label{higher}
 Let $D$ be a convex domain with nonempty boundary, $1<p<\infty$, $0<T\leq\infty$, $d-1-\alpha p/2<\theta<d-1+p+\alpha p/2$, and $\nu_t$ satisfy Assumption \ref{assum_t}. Suppose that $f\in \psi^{-\alpha/2}\bL_{p,\theta}(D,T)$, $u_0\in \psi^{\alpha/2-\alpha/p} B_{p,p;\theta}^{\alpha-\alpha/p}(D)$, and $u$ is a solution to \eqref{main_para} such that $u\in \psi^{\alpha/2}\bL_{p,\theta}(D,T)$. Then, $u\in \frH_{p,\theta}^\alpha(D,T)$, and
\begin{align} \label{eq6092336}
\|u\|_{\frH_{p,\theta}^{\alpha}(D,T)} &\leq N \|\psi^{-\alpha/2}u\|_{\bL_{p,\theta}(D,T)} + N \|\psi^{\alpha/2}f\|_{\bL_{p,\theta}(D,T)} \nonumber
\\
&\quad+ N \|\psi^{-\alpha/2+\alpha/p}u_0\|_{B_{p,p;\theta}^{\alpha-\alpha/p}(D)},
\end{align}
where $N=N(d,p,\alpha, \theta,\lambda,\Lambda,D)$.

Moreover, if $D$ is an open set with nonempty boundary, then the claim still holds provided that $d-\alpha p/2<\theta<d+\alpha p/2$.
\end{lem}

\begin{proof}
It suffices to assume that $T<\infty$.
We first prove that
\begin{align} \label{eq6101644}
\|\psi^{-\alpha/2}u\|_{\bH_{p,\theta}^{\alpha/2}(D,T)} &\leq N \|\psi^{-\alpha/2}u\|_{\bL_{p,\theta}(D,T)} + N \|\psi^{\alpha/2}f\|_{\bL_{p,\theta}(D,T)} \nonumber
\\
&\quad+ N \|\psi^{-\alpha/2+\alpha/p}u_0\|_{B_{p,p;\theta}^{\alpha-\alpha/p}(D)}.
\end{align}
Let
$$
u_n(t,x):=u(e^{n\alpha}t, e^n x), \quad f_n(t,x) := f(e^{n\alpha}t,e^n x), \quad u_{0n}(x):=u(e^n x), \quad n\in\bZ.
$$
Then, $u_n(t,x)\zeta_{-n}(e^nx)\in \bL_p(e^{-n\alpha}T):=L_p((0,e^{-n\alpha}T);L_p(\bR^d))$ is a weak solution to the equation
$$
\begin{cases}
\partial_t v=Lv+F_n,   \quad &(t,x)\in (0,e^{-n\alpha}T)\times \bR^d,
\\
v(0,\cdot)=u_{0n}(\cdot)\zeta_{-n}(e^n\cdot), \quad &x\in \bR^d,
\end{cases}
$$
where
\begin{align*}
F_n(t,x) &:= e^{n\alpha}(f_n(\cdot,\cdot)\zeta_{-n}(e^n\cdot))(t,x)
\\
&\quad - L(u_n(\cdot,\cdot)\zeta_{-n}(e^n\cdot))(t,x) + \zeta_{-n}(e^nx) Lu_n(t,x).
\end{align*}
Due to Lemma \ref{lem6091055} with $\gamma=-\alpha/2$,
\begin{align} \label{eq6100010}
&\sum_{n\in\bZ}e^{n(\theta-\alpha p/2)}\|F_n(e^{-n\alpha}t,\cdot)\|_{H_p^{-\alpha/2}}^p \nonumber
\\
&\leq N\|\psi^{-\alpha/2}u(t,\cdot)_D\|_{L_{p,\theta}(D)}^p + N\|\psi^{\alpha/2}f(t,\cdot)\|_{H_{p,\theta}^{-\alpha/2}(D)}^p.
     \end{align}
This implies that $F_n\in \bH_p^{-\alpha/2}(e^{-n\alpha}T) := L_p((0,e^{-n\alpha}T);H_p^{-\alpha/2})$. By Lemma \ref{lem_whole} with $\gamma=-\alpha/2$, we have $u_n(\cdot,\cdot)\zeta_{-n}(e^n\cdot)\in \bH_p^{\alpha/2}(e^{-n\alpha}T)$ and
\begin{align} \label{eq6101703}
&\|\Delta^{\alpha/2}(u(\cdot,e^n\cdot)\zeta_{-n}(e^n\cdot))\|_{\bH_p^{-\alpha/2}(T)}^p \nonumber
\\
&=e^{n\alpha}\|\Delta^{\alpha/2}(u_n(\cdot,\cdot)\zeta_{-n}(e^n\cdot))\|_{\bH_p^{-\alpha/2}(e^{-n\alpha}T)}^p   \nonumber
\\
&\leq N e^{n\alpha}\|F_{n}(\cdot,\cdot)\|_{\bH_p^{-\alpha/2}( e^{-n\alpha}T)}^p + N e^{n\alpha}\|u_{0n}(\cdot)\zeta_{-n}(e^n\cdot)\|_{B_p^{\alpha/2-\alpha/p}}^p \nonumber
\\
&= N \|F_{n}(e^{-n\alpha}\cdot,\cdot)\|_{\bH_p^{-\alpha/2}(T)}^p + N e^{n\alpha}\|u_{0n}(\cdot)\zeta_{-n}(e^n\cdot)\|_{B_p^{\alpha/2-\alpha/p}}^p.
\end{align}
By \eqref{eq6100010}, \eqref{eq6101703} and the relation
\begin{equation*}
\|v\|_{H_p^{\alpha/2}} \leq N \left(\|v\|_{H_p^{-\alpha/2}}+\|\Delta^{\alpha/2}v\|_{H_p^{-\alpha/2}}\right),
\end{equation*}
we have
\begin{align} \label{eq6211510}
&\|\psi^{-\alpha/2}u\|^p_{\bH^{\alpha/2}_{p,\theta}(D)} \nonumber
\\
& \leq N \sum_n e^{n(\theta-\alpha p/2)} \left(\|u(e^n\cdot)\zeta_{-n}(e^n\cdot)\|^p_{H^{-\alpha/2}_p}+ \|\Delta^{\alpha/2}(u(e^n\cdot)\zeta_{-n}(e^n\cdot))\|^p_{H_p^{-\alpha/2}}\right)   \nonumber
\\
&\leq N \|\psi^{-\alpha/2}u\|_{\bL_{p,\theta}(D,T)}^p + N \|\psi^{\alpha/2}f\|_{\bL_{p,\theta}(D,T)}^p + N \|\psi^{-\alpha/2+\alpha/p}u_0\|_{B_{p,p;\theta}^{\alpha-\alpha/p}(D)}^p,
\end{align}
which gives \eqref{eq6101644}.

Now we prove \eqref{eq6092336}. By repeating the above argument, one can obtain
\begin{align} \label{eq6231607}
\|\psi^{-\alpha/2}u\|_{\bH_{p,\theta}^{\alpha}(D,T)} &\leq N \|\psi^{-\alpha/2}u\|_{\bH_{p,\theta}^{\alpha/2}(D,T)} + N \|\psi^{\alpha/2}f\|_{\bL_{p,\theta}(D,T)} \nonumber
\\
&\quad+ N \|\psi^{-\alpha/2+\alpha/p}u_0\|_{B_{p,p;\theta}^{\alpha-\alpha/p}(D)}.
\end{align}
This together with \eqref{eq6101644} easily yields \eqref{eq6092336}. Moreover, by Lemma \ref{lem6101710}, $\partial_t u\in \psi^{-\alpha/2}\bL_{p,\theta}(D,T)$, and thus $u\in \frH_{p,\theta}^\alpha(D,T)$.
 The lemma is proved.
\end{proof}

In the following lemma, the corresponding result for the elliptic equations is obtained.

\begin{lem} \label{lem6211707}
  Let $D$ be a convex domain with nonempty boundary, $1<p<\infty$, $d-1-\alpha p/2<\theta<d-1+p+\alpha p/2$, and $\nu$ satisfy Assumption \ref{assum}. Suppose that $f\in \psi^{-\alpha/2}L_{p,\theta}(D)$, and $u$ is a solution to \eqref{main_ell} such that $u\in \psi^{\alpha/2}L_{p,\theta}(D)$. Then, $u\in \psi^{\alpha/2}H_{p,\theta}^\alpha(D)$, and
\begin{align} \label{eq6211412}
\|\psi^{-\alpha/2}u\|_{H_{p,\theta}^{\alpha}(D)} &\leq N \|\psi^{-\alpha/2}u\|_{L_{p,\theta}(D)} + N \|\psi^{\alpha/2}f\|_{L_{p,\theta}(D)},
\end{align}
where $N=N(d,p,\alpha, \theta,\lambda,\Lambda,D)$.

Moreover, if $D$ is an open set with nonempty boundary, then the claim still holds provided that $d-\alpha p/2<\theta<d+\alpha p/2$.
\end{lem}

\begin{proof}
As in the proof of Lemma \ref{higher}, we first show
\begin{align} \label{eq6211511}
\|\psi^{-\alpha/2}u\|_{H_{p,\theta}^{\alpha/2}(D)} &\leq N \|\psi^{-\alpha/2}u\|_{L_{p,\theta}(D)} + N \|\psi^{\alpha/2}f\|_{L_{p,\theta}(D)}.
\end{align}

For $n\in\bZ$, denote $u_n(x):=u(e^n x)$ and $f_n(x):=f(e^nx)$. Then, we have
\begin{equation} \label{eq6214157}
L(u_n(\cdot)\zeta_{-n}(e^n\cdot))(x)= F_n(x),\quad x\in \bR^d,
\end{equation}
where
\begin{align*}
F_n(x):=e^{n\alpha}f_n(x)\zeta_{-n}(e^nx)- \left(L(u_n(\cdot)\zeta_{-n}(e^n\cdot))(x)-\zeta_{-n}(e^nx) Lu_n(x)\right).
\end{align*}
   By Lemma \ref{lem6091055} with $\gamma=-\alpha/2$,
   \begin{align*}
     \sum_{n\in\bZ} e^{n(\theta-\alpha p/2)}\|F_n\|_{H_p^{-\alpha/2}} \leq N \|\psi^{-\alpha/2}u\|_{L_{p,\theta}(D)} + N \|\psi^{\alpha/2}f\|_{L_{p,\theta}(D)}.
   \end{align*}
   Thus, from \eqref{eq6211501} and \eqref{eq6214157},
   \begin{align*}
     \|\Delta^{\alpha/2}(u_n(\cdot)\zeta_{-n}(e^n\cdot))\|^p_{H_p^{\gamma-\alpha}} \leq N \|L(u_n(\cdot)\zeta_{-n}(e^n\cdot))\|^p_{H_p^{\gamma-\alpha}} = N\|F_n\|^p_{H_p^{\gamma-\alpha}}.
   \end{align*}
Thus, as in \eqref{eq6211510}, we have \eqref{eq6211511}.

Finally, as in \eqref{eq6231607}, one can obtain \eqref{eq6211412} by repeating the above argument. The lemma is proved.
\end{proof}

\mysection{Proof of Theorems \ref{thm_para} and \ref{thm_ell}} \label{sec_proof}

We first introduce a probabilistic representation of solution.
A rotationally symmetric $\alpha$-stable $d$-dimensional process $X=\{X_t, t\geq0\}$ is a L\'evy process defined on a probability space $(\Omega,\cF,\bP)$ such that
\begin{equation*}
   \bE e^{i\xi\cdot X_t}=e^{-|\xi|^\alpha t}, \quad \forall \xi\in\bR^d.
 \end{equation*}
For $x\in \bR^d$, let $\kappa_{D}:=\kappa^x_D:=\inf\{t\geq0: x+X_t\not\in D\}$ be the first exit time of $X$ from $D$.
For bounded measurable functions $f$, we denote
\begin{equation*}
  P_t^D f(x) = \bE [f(x+X_t); \kappa_D>t], \quad t>0, x\in \bR^d.
\end{equation*}
Obviously, $P_t^Df(x)=0$ for $x\in D^c$.
It is known that $\{P_t^D\}_{t\geq0}$ is a Feller semigroup in $L_\infty(D)$ if $D$ is a $C^{1,\tau}$ open set (see page 68 of \cite{Chung}).

\begin{lem} \label{lem6181520}
  Let $D$ be a bounded $C^{1,\tau}$ open set, $p\in(1,\infty)$, $T\in(0,\infty)$, $\alpha\in(0,2)$ and $\theta\in(d-1,\infty)$. Then, for any $u_0\in C_c^\infty(D)$ and $f\in C_c^\infty((0,T)\times D)$,
    \begin{equation*}
    u(t,x):=P_t^D u_0(x) + \int_0^t P_{t-s}^D f(s,\cdot)(x) \,ds
  \end{equation*}
  belongs to $\psi^{\alpha/2}\bL_{p,\theta}(D,T)$, and is a (weak) solution to \eqref{main_para} with $L=-(-\Delta)^{\alpha/2}$.
\end{lem}

\begin{proof}
First, one can show that  $u$ is a weak solution to \eqref{main_para} by following the proof of \cite[Lemma 8.4]{Zhang_Dirichlet}, which treats the case $u_0=0$. The general case can be handled similarly. Thus, it remains to show $u\in \psi^{\alpha/2}\bL_{p,\theta}(D,T)$.

Let $\widetilde{\psi}$ be a regularized distance, and take $\delta$ from Lemma \ref{lem_dist_open}. Due to the condition $\theta>d-1$, we take $\beta$ such that $\beta\in(0,\alpha/2)$ and
 \begin{equation} \label{eq6181436}
   \beta p-\alpha p/2+\theta-d>-1.
 \end{equation}
Notice that $u$ is bounded since $u_0$ and $f$ are bounded. Therefore, there exists sufficiently large $N_0>0$ (depending also on $u_0$ and $f$) such that $v_\beta(x):=N_0 \widetilde{\psi}^\beta(x)$ satisfies
\begin{equation*}
  |u_0(x)| \leq v_\beta(x), \quad x\in D,
\end{equation*}
\begin{equation*}
  |u(t,x)| \leq v_\beta(x), \quad t\in(0,T), \,  d_x\geq\delta,
\end{equation*}
and
\begin{equation*}
  Lv_\beta(x) \leq -N N_0 \widetilde{\psi}^{\beta-\alpha}(x) \leq -|f(t,x)|, \quad t\in(0,T), \,  d_x<\delta.
\end{equation*}
Thus, from the last inequality above, we have
\begin{equation*}
  (u-v_\beta)_t-L(u-v_\beta)\leq 0, \quad t\in(0,T), \,  d_x<\delta.
\end{equation*}
 By applying the maximum principle to $u-v_\beta$ over $(0,T)\times\{0<d_x<\delta\}$ (see e.g. \cite[Theorem 3.2]{LD}), $u\leq v_\beta$ for $x\in D$. Using the same argument for $-u$, we conclude $|u|\leq v_\beta$. Thus, by \eqref{eq6181436}, $u\in \psi^{\alpha/2}\bL_{p,\theta}(D,T)$. The lemma is proved.
\end{proof}

Here we deal with representation of solution to the elliptic equations.
\begin{lem} \label{lem6211708}
    Let $D$ be a bounded $C^{1,\tau}$ open set, $p\in(1,\infty)$, $\alpha\in(0,2)$ and $\theta\in(d-1,\infty)$. Then, for any $f\in C_c^\infty( D)$,
    \begin{equation*}
    u(x):=-\int_0^\infty P_t^D f(x) \,dt
  \end{equation*}
  belongs to $\psi^{\alpha/2}L_{p,\theta}(D)$, and is a (weak) solution to \eqref{main_ell} with $L=-(-\Delta)^{\alpha/2}$.
\end{lem}

\begin{proof}
By \cite[Lemma 5.9]{Zhang_Dirichlet}, if $D$ is a bounded domain, then $\bE(\kappa_D^x)^2\leq N$, where $N$ is independent of $x$. Here, we remark that \cite{Zhang_Dirichlet} is proved only for domains, but the estimate still holds even for open sets. Hence, using this,
\begin{equation*}
    |P_t^Df(x)|\leq \|f\|_{L_\infty(D)} \bP(\kappa_D^x<t) \leq \|f\|_{L_\infty(D)} \min\{1, \frac{1}{t^2}\bE(\kappa_D^x)^2\},
  \end{equation*}
  which implies that $u$ is well defined, and $P_t^Df(x)\to0$ uniformly with respect to $x$ as $t\to\infty$.
  Thus, one can show that $u$ is a solution to \eqref{main_ell} by repeating the proof of \cite[Lemma 3.4 $(ii)$]{Dirichlet}.

Now we prove $u\in\psi^{\alpha/2}L_{p,\theta}(D)$.
As is the proof of Lemma \ref{lem6181520}, we take a regularized distance $\widetilde{\psi}$ and $\delta$ from Lemma \ref{lem_dist_open}. Then, for $\beta\in(0,\alpha/2)$ satisfying \eqref{eq6181436}, there exists $N_0>0$ such that $v_\beta(x):=N_0 \widetilde{\psi}^\beta(x)$ satisfies
\begin{equation*}
  |u(x)| \leq v_\beta(x), \quad  d_x\geq\delta,
\end{equation*}
and
\begin{equation*}
  Lv_\beta(x) \leq -N N_0 \widetilde{\psi}^{\beta-\alpha}(x) \leq -|f(x)|, \quad d_x<\delta.
\end{equation*}
Hence, applying the maximum principle to $u-v_\beta$ over $\{0<d_x<\delta\}$ (see e.g. \cite[Theorem 5.2]{CS}), we have $|u|\leq v_\beta$. Thus, from \eqref{eq6181436}, we have $u\in\psi^{\alpha/2}L_{p,\theta}(D)$. The lemma is proved.
\end{proof}

\begin{proof}[Proof of Theorem \ref{thm_para}]
Note that the case $T=\infty$ can be easily treated if the theorem is proved for any $T<\infty$. Thus, we assume $T<\infty$.

\textbf{1.} Assume that $D=\bR_+^d$.
 Since $C_c^\infty([0,T]\times D)$ is dense in $\frH_{p,\theta}^\alpha(D,T)$ (see \cite[Remark 5.5]{Krylovhalf}), the a priori estimate \eqref{apriori} easily follows from Lemmas \ref{lem6151427} and \ref{higher}.
Next, the solvability of \eqref{main_para} with $L_t=-(-\Delta)^{\alpha/2}$ is treated in \cite[Theorems 2.2 and 2.9]{Dirichlet}.
Thus, thanks to the method of continuity, we obtain the solvability for general operators.

\textbf{2.} Suppose that $D$ is bounded and $\theta\leq d+\alpha p/2-\alpha$. As in the case $1$, Lemmas \ref{lem_Hardy_dom} and \ref{higher} yield the a priori estimate \eqref{apriori}.
Thus, again by the method of continuity, we only need to show the solvability of \eqref{main_para} with $L=-(-\Delta)^{\alpha/2}$.
By Lemmas \ref{higher} and \ref{lem6181520}, for $u_0\in C_c^\infty(D)$ and $f\in C_c^\infty((0,T)\times D)$, there exists a solution $u\in \frH_{p,\theta}^\alpha(D,T)$ to \eqref{main_para}.
For general $u_0\in\psi^{\alpha/2-\alpha/p}B_{p,p;\theta}^{\alpha-\alpha/p}(D)$ and $f\in \psi^{-\alpha/2}\bL_{p,\theta}(D,T)$, take a sequence $u_{0n}\in C_c^\infty(D)$ and $f_n\in C_c^\infty((0,T)\times D)$ such that $u_{0n}\to u$ and $f_n\to f$ in $\psi^{\alpha/2-\alpha/p}B_{p,p;\theta}^{\alpha-\alpha/p}(D)$ and $\psi^{-\alpha/2}\bL_{p,\theta}(D,T)$, respectively. For each $n$, $u_n$ denotes a solution to \eqref{main_para} with $u_n$ and $f_n$, in place of $u$ and $f$, respectively. Then, by \eqref{apriori},
\begin{align*}
&\|(u_n-u_m)\|_{\frH_{p,\theta}^{\alpha}(D,T)}
\\
&\leq N \|\psi^{-\alpha/2+\alpha/p}(u_{0n}-u_{0m})\|_{B_{p,p;\theta}^{\alpha-\alpha/p}(D)} + N \|\psi^{\alpha/2}(f_n-f_m)\|_{\bL_{p,\theta}(D,T)},
\end{align*}
which actually implies $u_n$ is Cauchy in $\frH_{p,\theta}^{\alpha}(D,T)$. Hence, the limit of the sequence, say $u$, is a solution to \eqref{main_para}, and $u\in\frH_{p,\theta}^{\alpha}(D,T)$.

\textbf{3.} Lastly, we assume $D$ is bounded and $\theta> d+\alpha p/2-\alpha$. We will use a duality argument to prove the a priori estimate. Let $u,v\in C_c^\infty([0,T]\times D)$. Then, by integration by parts,
\begin{align*}
  \int_0^T \int_D u(\partial_t v+Lv) \,dxdt &= - \int_0^T \int_D v(\partial_t u-Lu) \,dxdt
  \\
  &\quad+ \int_D u(T)v(T) \,dx - \int_D u(0)v(0) \,dx.
\end{align*}
By Lemma \ref{lem_prop} $(iii)$,
\begin{align*}
  &\left|\int_0^T \int_D u(\partial_t v+Lv) \,dxdt\right|
  \\
  &\leq N \|\psi^{\alpha/2}(\partial_t u-Lu)\|_{\bL_{p,\theta}(D,T)}\|\psi^{-\alpha/2}v\|_{\bL_{p',\theta'}(D,T)}
  \\
  &\quad+ N \|\psi^{-\alpha/2+\alpha /p}u(0)\|_{L_{p,\theta}(D)} \|\psi^{-\alpha/2+\alpha /p'}v(0)\|_{L_{p',\theta'}(D)}
  \\
  &\quad+ N \|\psi^{-\alpha/2+\alpha /p}u(T)\|_{L_{p,\theta}(D)} \|\psi^{-\alpha/2+\alpha /p'}v(T)\|_{L_{p',\theta'}(D)}.
  \end{align*}
where $1/p+1/p'=1$ and $\theta/p+\theta'/p' = d$.
Since $\theta'<d+\alpha p'/2 -\alpha$ , by applying Lemma \ref{lem_Hardy_dom} to $v(T-t,x)$, we have
\begin{align*}
  &\|\psi^{-\alpha/2+\alpha /p'}v(0)\|_{L_{p',\theta'}(D)}
  \\
  &\leq N \|\psi^{\alpha/2}(\partial_t v+Lv)\|_{\bL_{p',\theta'}(D,T)} + N\|\psi^{-\alpha/2+\alpha /p'}v(T)\|_{L_{p',\theta'}(D)}.
\end{align*}
Moreover, by the result for the case $2$, for any $g\in \psi^{-\alpha/2}\bL_{p,\theta}(D,T)$, one can find a solution $w$ to
\begin{equation*}
\begin{cases}
\partial_t w(t,x)=-Lw(t,x)+g(t,x),\quad &(t,x)\in(0,T)\times D,
\\
w(T,x)=0,\quad & x\in D,
\\
w(t,x)=0, \quad & (t,x)\in(0,T)\times D^c,
\end{cases}
\end{equation*}
satisfying $w\in \bH_{p,\theta}^\alpha(D,T)$, and
\begin{equation*}
  \|\psi^{-\alpha/2}w\|_{\bH_{p,\theta}^\alpha(D,T)} \leq N \|\psi^{\alpha/2} g\|_{\bL_{p,\theta}(D,T)}.
\end{equation*}
Here, note that the initial condition is defined at $t=T$ and the sign of the operator is reversed.
Thus, this solvability and the denseness of $\psi^{-\alpha/2}\bL_{p,\theta}(D,T)$ actually imply that for any $g\in \psi^{-\alpha/2}\bL_{p,\theta}(D,T)$, there exists $v_n\in C_c^\infty([0,T]\times D)$ such that $v_n(T)\to0$, $v_n\to w$ and $\partial_t v_{n} +Lv_n \to g$ in their corresponding spaces.
Therefore, for any $g\in \psi^{-\alpha/2}\bL_{p,\theta}(D,T)$,
\begin{align*}
  \left|\int_0^T \int_D ug \,dxdt\right| &\leq N \|\psi^{\alpha/2}(u_t-Lu)\|_{\bL_{p,\theta}(D,T)} \|\psi^{\alpha/2} g\|_{\bL_{p,\theta}(D,T)}
  \\
  &\quad+ N \|\psi^{-\alpha/2+\alpha /p}u(0)\|_{L_{p,\theta}(D)} \|\psi^{\alpha/2} g\|_{\bL_{p,\theta}(D,T)}.
\end{align*}
Thus, by Lemma \ref{lem_prop} $(iii)$,
\begin{align*}
  \|\psi^{-\alpha/2}u\|_{\bL_{p,\theta}(D,T)} \leq N \|\psi^{\alpha/2}(u_t-Lu)\|_{\bL_{p,\theta}(D,T)} + N\|\psi^{-\alpha/2+\alpha /p}u(0)\|_{L_{p,\theta}(D)}.
\end{align*}
This together with Lemma \ref{higher} yields the a priori estimate \eqref{apriori}. Now one can obtain the solvability of \eqref{main_para} by repeating the argument used in the step $2$. The theorem is proved.
\end{proof}

\begin{proof}[Proof of Theorem \ref{thm_ell}]
  We first prove the a priori estimate \eqref{aprioriell} by following the proof of \cite[Theorem 2.6]{KryVMO}. Here, by Lemma \ref{lem_prop} $(i)$, we only need to prove \eqref{aprioriell} for $u\in C_c^\infty(D)$.

  Let $\eta \in C_c^\infty((0,\infty))$ and $u\in C_c^\infty(D)$. Then, $v(t,x):=\eta(t/n)u(x)$ satisfies
\begin{equation*}
\begin{cases}
\partial_t v(t,x)=Lv(t,x)+g(t,x),\quad &(t,x)\in(0,\infty)\times D,
\\
v(0,x)=0,\quad & x\in D,
\\
v(t,x)=0, \quad & (t,x)\in(0,\infty)\times D^c,
\end{cases}
\end{equation*}
where $g(t,x):=\frac{1}{n}\eta'(t/n)u(x) - \eta(t/n)Lu(x)$.
Observe that
\begin{align*}
  \|\psi^{-\alpha/2}v\|_{\bH_{p,\theta}^\alpha(D,\infty)}^p=nN_1\|\psi^{-\alpha/2}u\|_{H_{p,\theta}^\alpha(D)}^p,
\end{align*}
and
\begin{align*}
  \|\psi^{\alpha/2}g\|_{\bL_{p,\theta}(D,\infty)}^p=N\left(nN_1 \|\psi^{\alpha/2}Lu\|_{L_{p,\theta}(D)} + n^{1-p}N_2 \|\psi^{\alpha/2}u\|_{L_{p,\theta}(D)}\right),
\end{align*}
where
\begin{equation*}
  N_1:=\int_0^\infty |\eta|^p \,dt, \quad N_2:=\int_0^\infty |\eta'|^p \,dt.
\end{equation*}
This and \eqref{apriori} with $T=\infty$ yield \eqref{aprioriell}.

Due to the method of continuity, we only need to prove the solvability of \eqref{main_ell} with $L=-(-\Delta)^{\alpha/2}$. For the case $D=\bR_+^d$, see \cite[Theorems 2.3 and 2.10]{Dirichlet}. Now we consider the general open sets. For $f\in C_c^\infty(D)$, Lemmas \ref{lem6211707} and \ref{lem6211708} easily lead to the solvability. Then, the standard approximation argument as in the proof of Theorem \ref{thm_para} yields the desired result. The theorem is proved.
\end{proof}

\appendix

\mysection{One dimensional distance functions}

\begin{lemma} \label{lem_a2}
Let $d=1$ and $L$ be an operator of the form \eqref{opd1} with $\alpha\neq1$. Let
\begin{equation*}
u(x):=(x_+)^\beta, \quad \beta\in(-1,\alpha).
\end{equation*}
 Then,
 \begin{equation*}
Lu(x)=K_{\alpha,\beta}(x_+)^{\beta-\alpha}, \quad x\in \bR_+,
 \end{equation*}
 where
  \begin{equation*}
K_{\alpha,\beta}=-\frac{2}{\pi}\Gamma(-\alpha)\Gamma(1+\beta)\Gamma(\alpha-\beta)\cos(\alpha \pi/2)\sin((\beta-\alpha/2)\pi).
 \end{equation*}
 In particular,
 \begin{align} \label{eq_a12}
\begin{cases}
K_{\alpha,\beta}>0, \quad \beta\in(-1,-1+\alpha/2)\cup(\alpha/2,\alpha),
\\
K_{\alpha,\beta}=0, \quad \beta=-1+\alpha/2 \text{ or } \alpha/2,
\\
K_{\alpha,\beta}<0, \quad \beta\in(-1+\alpha/2,\alpha/2).
\end{cases}
 \end{align}
\end{lemma}

\begin{proof}
We first assume that $\alpha\in(0,1)$. Using Euler's reflection formula
\begin{equation*}
  \Gamma(z)\Gamma(1-z)=\frac{\pi}{\sin(\pi z)}, \quad z\notin \bZ,
\end{equation*}
 one can easily prove the case $\beta=0$. Thus, we only consider $\beta\neq0$.
By the change of variables, for $x>0$,
\begin{align*}
Lu(x)&= \int_{-\infty}^\infty ( (x+y)_{+}^{\beta} - (x_+)^{\beta} ) |y|^{-1-\alpha} dy
\\
&=\left(\int_{-\infty}^\infty ( (1+y)_{+}^{\beta} - 1 ) |y|^{-1-\alpha} dy\right) (x_+)^{\beta-\alpha} =:M_{\alpha,\beta} (x_+)^{\beta-\alpha}.
\end{align*}
Here,
\begin{align} \label{eq_a4}
M_{\alpha,\beta} &= -\int_{-\infty}^{-1} |y|^{-1-\alpha}dy + \int_{-1}^0 ( (1+y)^{\beta} - 1 ) |y|^{-1-\alpha} dy \nonumber
\\
&\quad+ \int_{0}^\infty ( (1+y)^{\beta} - 1 ) |y|^{-1-\alpha} dy \nonumber
\\
&=-\alpha^{-1} + \int_0^1 ( (1-y)^{\beta} - 1 ) y^{-1-\alpha} dy+ \int_{0}^\infty ( (1+y)^{\beta} - 1 ) y^{-1-\alpha} dy \nonumber
\\
&=:-\alpha^{-1}+I_1(\alpha,\beta)+I_2(\alpha,\beta).
\end{align}
By the definition of the beta function $B(a,b)$,
\begin{align} \label{eq_a1}
I_1(\alpha,\beta)-I_1(\alpha,\beta+1)&=\int_0^1 \left((1-y)^{\beta}-(1-y)^{\beta+1}\right)y^{-1-\alpha}dy \nonumber
\\
&= \int_0^1 (1-y)^{\beta}y^{-\alpha} dy = B(1-\alpha,\beta+1).
\end{align}
Since $\beta+1>0$, we can use the Fubini theorem to get
\begin{align} \label{eq_a2}
I_1(\alpha,\beta+1) &= - (\beta+1)\int_0^1 \int_{1-y}^1 y^{-1-\alpha}z^\beta dzdy \nonumber
\\
&=\alpha^{-1}(\beta+1) \int_0^1 (1-(1-z)^{-\alpha})z^\beta dz \nonumber
\\
&= \alpha^{-1}-\alpha^{-1}(\beta+1)B(1-\alpha,\beta+1).
\end{align}
Combining \eqref{eq_a1} and \eqref{eq_a2},
\begin{equation} \label{eq_a3}
I_1(\alpha,\beta)=\alpha^{-1}+\alpha^{-1}(\alpha-\beta-1)B(1-\alpha,\beta+1).
\end{equation}
Now we consider $I_2(\beta)$. By the fundamental theorem of calculus and the Fubini theorem,
\begin{align} \label{eq_a5}
I_2(\alpha,\beta)&=\beta\int_0^\infty \int_0^y (1+z)^{\beta-1}y^{-1-\alpha} dzdy \nonumber
\\
&= \alpha^{-1}\beta\int_0^\infty (1+z)^{\beta-1} z^{-\alpha} dz \nonumber
\\
&= \alpha^{-1}\beta B(1-\alpha,\alpha-\beta).
\end{align}
Here, for the last equality, we used a well-known formula
\begin{equation} \label{eq_ap7}
B(a,b)=\int_0^\infty t^{a-1} (1+t)^{-a-b} \,dt.
\end{equation}
Thus, by Euler's reflection formula and equalities \eqref{eq_a4}, \eqref{eq_a3} and \eqref{eq_a5},
\begin{align} \label{eq_ap8}
M_{\alpha,\beta} &= \alpha^{-1}(\alpha-\beta-1)B(1-\alpha,\beta+1) + \alpha^{-1}\beta B(1-\alpha,\alpha-\beta) \nonumber
\\
&=-\frac{\Gamma(1-\alpha)}{\alpha} \left(\frac{\Gamma(\beta+1)}{\Gamma(1-\alpha+\beta)}+\frac{\Gamma(\alpha-\beta)}{\Gamma(-\beta)} \right) \nonumber
\\
&= -\frac{\pi\Gamma(1-\alpha)}{\alpha\Gamma(1-\alpha+\beta)\Gamma(-\beta)} \left( \frac{1}{\sin((\alpha-\beta)\pi)} - \frac{1}{\sin(\beta\pi)} \right) \nonumber
\\
&=-\frac{\pi\Gamma(1-\alpha)}{\alpha\Gamma(1-\alpha+\beta)\Gamma(-\beta)}\left( \frac{2\cos(\alpha \pi/2)\sin((\beta-\alpha/2)\pi)}{\sin((\alpha-\beta)\pi)\sin(\beta\pi)} \right) \nonumber
\\
&= K_{\alpha,\beta}.
\end{align}
Hence, the case $\alpha\in(0,1)$ is proved.

Next, we consider $\alpha\in(1,2)$. As above, we assume $\beta\neq1$. By the change of variables,
\begin{align*}
Lu(x)= M_{\alpha,\beta} (x_+)^{\beta-\alpha}, \quad x>0.
\end{align*}
where
\begin{align} \label{eq_ap9}
M_{\alpha,\beta} &:= \int_{-\infty}^\infty ( (1+y)_{+}^{\beta} - 1 -\beta y ) |y|^{-1-\alpha} dy \nonumber
\\
&= \int_{-\infty}^{-1} \cdots + \int_{-1}^{0} \cdots + \int_{0}^{\infty} \cdots \nonumber
\\
&= -\alpha^{-1} +\beta(\alpha-1)^{-1} + \int_{-1}^{0} \cdots + \int_{0}^{\infty} \cdots \nonumber
\\
&=: -\alpha^{-1} +\beta(\alpha-1)^{-1} + J_1(\alpha,\beta) + J_2(\alpha,\beta).
\end{align}
As in \eqref{eq_a1},
\begin{align} \label{eq_a6}
J_1(\alpha,\beta) - J_1(\alpha,\beta+2) &= J_1(\alpha,\beta) - J_1(\alpha,\beta+1) + J_1(\alpha,\beta+1) - J_1(\alpha,\beta+2) \nonumber
\\
&= I_1(\alpha-1,\beta) + I_1(\alpha-1,\beta+1) \nonumber
\\
&= 2(\alpha-1)^{-1} + (\alpha-1)^{-1}(\alpha-\beta-2) B(2-\alpha,\beta+1) \nonumber
\\
&\quad+ (\alpha-1)^{-1}(\alpha-\beta-3) B(2-\alpha,\beta+2).
\end{align}
Here, for the last equality, we used \eqref{eq_a3}. By the fundamental theorem of calculus and the Fubini theorem,
\begin{align*}
J_1(\alpha,\beta+2) &= (\beta+2)(\beta+1) \int_0^1 \int_0^y \int_0^z  (1-t)^{\beta}y^{-1-\alpha} \,dtdzdy
\\
&= (\beta+2)(\beta+1) \int_0^1 \int_t^1 \int_z^1 (1-t)^{\beta}y^{-1-\alpha}\, dydzdt
\\
&= \frac{(\beta+2)(\beta+1)}{\alpha} \int_0^1 (1-t)^{\beta} \left(\frac{\alpha}{1-\alpha} - \frac{t^{1-\alpha}}{1-\alpha} + t \right)\,dt
\\
&= \frac{\beta+2}{1-\alpha} - \frac{(\beta+2)(\beta+1)}{\alpha(1-\alpha)} B(2-\alpha,\beta+1)
\\
&\quad+ \frac{(\beta+2)(\beta+1)}{\alpha} B(2,\beta+1).
\end{align*}
Combining this and \eqref{eq_a6},
\begin{equation} \label{eq_a10}
J_1(\alpha,\beta)=\frac{1}{\alpha}+\frac{\beta}{1-\alpha} - \frac{\Gamma(-\alpha) \Gamma(\beta+1)}{\Gamma(1-\alpha+\beta)}.
\end{equation}
For $J_2(\alpha,\beta)$, by the fundamental theorem of calculus and \eqref{eq_ap7},
\begin{align} \label{eq_a11}
J_2(\alpha,\beta) &= \beta(\beta-1) \int_0^\infty \int_0^y \int_0^z (1+t)^{\beta-2} y^{-1-\alpha} \,dtdzdy \nonumber
\\
&= -\frac{\beta(\beta-1)}{\alpha(1-\alpha)} \int_0^\infty (1+t)^{\beta-2} t^{1-\alpha} \,dt \nonumber
\\
&=-\frac{\beta(\beta-1)}{\alpha(1-\alpha)} B(2-\alpha, \alpha-\beta) \nonumber
\\
&= -\frac{\Gamma(1-\alpha)\Gamma(\alpha-\beta)}{\alpha \Gamma(-\beta)}.
\end{align}
Thus, as in \eqref{eq_ap8}, Euler's reflection formula and equalities \eqref{eq_ap9}, \eqref{eq_a10} and \eqref
{eq_a11} lead to
\begin{equation*}
M_{\alpha,\beta}=K_{\alpha,\beta}.
\end{equation*}
The lemma is proved.
\end{proof}

\begin{lemma} \label{lem_a3}
Let $d=1$ and $L=-(-\Delta)^{1/2}$. Let
\begin{equation*}
u(x):=(x_+)^\beta, \quad \beta\in(-1,1).
\end{equation*}
 Then,
 \begin{equation*}
Lu(x)=K_{1,\beta}(x_+)^{\beta-1}, \quad x\in \bR_+,
 \end{equation*}
 where
  \begin{align*}
K_{1,\beta}=
\begin{cases}
-\beta \cos(\beta \pi), \quad &\beta\in(0,1),
\\
-{1}/{\pi}, \quad &\beta=0,
\\
\beta \cos(\beta\pi), \quad &\beta\in(-1,0).
\end{cases}
 \end{align*}
 Moreover, \eqref{eq_a12} still holds true with $\alpha=1$.
\end{lemma}

\begin{proof}
See \cite[Proposition 4.4]{DRSV22} for the case $\beta\in(0,1)$. Now we consider $\beta<0$. Let
\begin{equation*}
v(x):=\frac{1}{\beta+1} (x_+)^{\beta+1}.
\end{equation*}
Then, since $\beta+1>0$, for $x>0$,
\begin{equation*}
Lu(x)=\frac{dLv}{dx}(x) = \frac{-(\beta+1)\cos((\beta+1)\pi))}{\beta+1} \frac{d}{dx}(x_+)^{\beta}= \beta\cos(\beta\pi)(x_+)^{\beta-1}.
\end{equation*}
Lastly, the case $\beta=0$ can be easily obtained from \eqref{opd1}. The lemma is proved.
\end{proof}

\mysection{Parabolic equations in the whole space} \label{sec_whole}

In this section, we present the $L_p$-maximal regularity of nonlocal parabolic equations in the whole space.
We consider more general operators than in the main sections above.

We first impose the assumption on a family of L\'evy measures $(\nu_t)_{t\in(0,T)}$.

\begin{assumption} \label{assum_whole}
  $(i)$ If $f$ is integrable with respect to $\nu_t$ for all $t\in(0,T)$, then the mapping
  $$
t\rightarrow \int_{\bR^d} f(y)\,\nu_t(dy)
  $$
  is measurable.

  $(ii)$ For any $\sigma>\alpha$,
  \begin{equation*}
    \int_{|y|\leq1} |y|^{\sigma}\, \nu_t(dy)<\infty.
  \end{equation*}

$(ii)$ There exist $\alpha_1$, $\alpha_2$ and $N_0>0$ such that for any $R>0$,
\begin{equation*}
  R^{\alpha-\alpha_1} \int_{|y|\leq R} |y|^{\alpha_1}\, \nu_t(dy) + R^{\alpha-\alpha_2} \int_{|y|> R} |y|^{\alpha_2} \,\nu_t(dy) \leq N_0,
\end{equation*}
where $\alpha_1,\alpha_2\in(0,1]$ if $\alpha\in(0,1)$; $\alpha_1,\alpha_2\in(1,2]$ if $\alpha\in(1,2)$; $\alpha_1\in(1,2]$ and $\alpha_2\in[0,1)$ if $\alpha=1$.

  $(iii)$ If $\alpha\in(1,2)$, then
  \begin{equation} \label{eq9031436}
    \int_{|y|>1} |y| \,\nu_t(dy)<\infty.
  \end{equation}

  $(iv)$ We have
  \begin{equation*}
    \sup_{t\in(0,T)} \int_{\bR^d} \min\{1,|y|^2\} \,\nu_t(dy) < \infty.
  \end{equation*}
  \end{assumption}

Let $\nu_t$ satisfy Assumption \ref{assum_whole} for some $\alpha\in(0,2)$. Now we define the nonlocal operator $L_t$ as
\begin{equation} \label{op_whole}
  L_tu:=\int_{\bR^d} \left( u(x+y)- u(x)- y^{(\alpha)}\cdot \nabla u(x)  \right) \,\nu_t(dy),
\end{equation}
where $y^{(\alpha)}:=\left(1_{1<\alpha<2} + 1_{\alpha=1} 1_{|y|\leq1}\right) y$.
  Here, due to \eqref{eq9031436}, \eqref{op_whole} is well defined for any $u\in C_b^2(\bR^d)$. We also denote the adjoint operator
  \begin{equation*}
  L_t^* u:=\int_{\bR^d} \left( u(x+y)- u(x)- y^{(\alpha)}\cdot \nabla u(x)  \right) \nu_t(-dy).
\end{equation*}

For the nondegeneracy of the operator, in Lemma \ref{lem_whole} below, we will assume that $\nu^{(1)} \leq \nu_t$ for some (nonsymmetric) $\alpha$-stable L\'evy measure $\nu^{(1)}$.
Here, we state the assumption on $\nu^{(1)}$.

\begin{assumption} \label{nu_stable}
  $(i)$ There exist $\lambda,\Lambda>0$ such that
  \begin{equation*}
    \lambda\leq \inf_{\rho\in S^{d-1}} \int_{S^{d-1}} |\rho\cdot \theta|^\alpha \mu^{(1)}(d\theta)
  \end{equation*}
  and
  \begin{equation*}
    \int_{S^{d-1}} \mu^{(1)}(d\theta) \leq\Lambda<\infty,
  \end{equation*}
  where $\mu^{(1)}$ is the spherical part of $\nu^{(1)}$.

  $(ii)$ When $\alpha=1$,
  \begin{equation} \label{eq9031521}
    \int_{S^{d-1}} \theta\mu^{(1)}(d\theta)=0.
  \end{equation}
\end{assumption}
Note that \eqref{eq9031521} is equivalent to
\begin{equation*}
  \int_{r<|y|<R} y \nu^{(1)}(dy)=0, \quad 0<r<R<\infty.
\end{equation*}

\begin{lem} \label{lem_whole}
  Let $\alpha\in(0,2)$, $1<p<\infty$, $0<T<\infty$, $\gamma\in\bR$, and $\nu^{(1)}$ be an $\alpha$-stable L\'evy measure. Suppose that L\'evy measures $\nu_t$ and $\nu^{(1)}$ satisfy Assumptions \ref{assum_whole} and \ref{nu_stable}, respectively. Assume that
\begin{equation*}
  \nu^{(1)} \leq \nu_t,
\end{equation*}
   $f\in L_p((0,T);H_p^\gamma)$ and $u_0\in B_{p,p}^{\gamma+\alpha-\alpha/p}$. Then, there exists a unique solution $u\in L_p((0,T);H_p^{\gamma+\alpha}) \cap L_\infty((0,T);H_p^\gamma)$ to
  \begin{equation} \label{eq_whole}
\begin{cases}
\partial_t u(t,x)=L_tu(t,x)+f(t,x),\quad &(t,x)\in(0,T)\times \bR^d,
\\
u(0,x)=u_0(x),\quad & x\in \bR^d.
\end{cases}
\end{equation}
More precisely, for any $\phi\in C_c^\infty(\bR^d)$ and $t\in(0,T)$,
\begin{equation*}
  (u(t,\cdot),\phi)_{\bR^d} = (u_0,\phi)_{\bR^d} + \int_0^t (u(s,\cdot),L_s^*\phi)_{\bR^d} \,ds + \int_0^t (f(s,\cdot),\phi)_{\bR^d} \,ds.
\end{equation*}
Moreover, for this solution $u$,
\begin{equation} \label{ineq_whole}
  \|(-\Delta)^{\alpha/2}u\|_{L_p((0,T);H_p^{\gamma})} \leq N \left(\|u_0\|_{B_{p,p}^{\gamma+\alpha-\alpha/p}} + \|f\|_{L_p((0,T);H_p^\gamma)}\right),
\end{equation}
where $N=N(d,p,\alpha,\gamma,\lambda,\Lambda,N_0)$ is independent of $u$ and $T$.
\end{lem}

\begin{proof}
  \textbf{1.} Since the isometry $(1-\Delta)^{\kappa_0/2}:H_p^{\kappa+\kappa_0}\to H_p^{\kappa}$ commutes with $L_t$, we only need to prove the claim for $\gamma=0$.

  \textbf{2.} Suppose that $\nu_t=\nu^{(1)}$ for all $t\in(0,T)$. In this case, one can check that $\nu^{(1)}$ satisfies all the assumptions in \cite[Theorem 1]{Mikul Cauchy}. Thus, by \cite[Theorem 1]{Mikul Cauchy}, we have $u\in L_p((0,T);H_p^{\gamma+\alpha})$ together with \eqref{ineq_whole}. Here, we note that the constant $N$ depends only on $d,p,\alpha,\lambda$, and $\Lambda$. Thus, it remains to show $u\in L_\infty((0,T);L_p)$.

  For general functions $h$, denote
$$
h^\varepsilon(x)=h*\Phi^\varepsilon(x), \quad \Phi^\varepsilon(x):=\varepsilon^{-d}\Phi(x/\varepsilon),
$$
where $\Phi$ is a standard mollifier on $\bR^d$, and $*$ denotes the convolution. Then, one can easily find
\begin{equation*}
  u^\varepsilon(t,x)=u_0^\varepsilon(x) + \int_0^t L_s u^\varepsilon(s,x)\,ds + \int_0^t f^\varepsilon(s,x)\,ds, \quad t\in(0,T),\, x\in \bR^d.
\end{equation*}
By the Minkowski inequality and H\"older's inequality, we have
\begin{align*}
  \|u^\varepsilon(t,\cdot)\|_{L_p} &\leq \|u_0^\varepsilon\|_{L_p} + \int_0^t \|L_s u^\varepsilon(s,\cdot)\|_{L_p} \,ds + \int_0^t \|f^\varepsilon(s,\cdot)\|_{L_p} \,ds
  \\
&\leq \|u_0^\varepsilon\|_{L_p} + T^{(p-1)/p} \|L_s u^\varepsilon\|_{L_p((0,T);L_p)} + T^{(p-1)/p}\|f^\varepsilon\|_{L_p((0,T);L_p)}.
\end{align*}
Letting $\varepsilon\downarrow0$, due to \eqref{ineq_whole} and the continuity of $L_t$ (see \cite[Lemma 14]{Mikul Cauchy}),
\begin{align} \label{eq9051621}
  \|u\|_{L_\infty((0,T);L_p)} \leq N(T) \left(\|u_0\|_{B_{p,p}^{\alpha-\alpha/p}} + \|f\|_{L_p((0,T);L_p)}\right).
\end{align}
Thus, we have $u\in L_\infty((0,T);L_p)$.

  \textbf{3.}
  Now we deal with the existence and \eqref{ineq_whole} for general $\nu_t$.
To consider time-dependent nonlocal operators, we use probabilistic arguments.
Let $\tilde{\nu}_t:=\nu_t-\nu^{(1)}$, and $p(dt,dy)$ be the Poisson random measure on $[0,T)\times\bR^d$ with intensity measure $\tilde{\nu}_t(dy)dt$. For the compensated Poisson random measure $q(dt,dy):=p(dt,dy)-\tilde{\nu}_t(dy)dt$, we define the stochastic process with independent increments
\begin{equation*}
 Y_{t}:=\int_0^t\int_{\bR^d} y^{(\alpha)} q(dr,dy) +\int_0^t\int_{\bR^d} \left(y-y^{(\alpha)}\right) p(dr,dy), \quad 0\leq t < T.
\end{equation*}

Let $f\in C_c^\infty((0,T)\times \bR^d)$ and $u_0\in C_c^\infty(\bR^d)$. By
the case \textbf{2}, there exists a solution $v$ satisfying
  \begin{equation*}
\begin{cases}
\partial_t v(t,x)=L^{\nu^{(1)}} v(t,x)+f(t,x-Y_t),\quad &(t,x)\in(0,T)\times \bR^d,
\\
u(0,x)=u_0(x),\quad & x\in \bR^d,
\end{cases}
\end{equation*}
where $L^{\nu^{(1)}}$ is the operator associated to $\nu^{(1)}$. Note that, by the representation formula of solution (see $(4.9)$ of \cite{Mikul Cauchy}), $v$ is well defined (measurable) on the probability space where $Y_t$ is defined. Since $f$ and $u_0$ are smooth, we have $v(t,\cdot)\in C^2(\bR^d)$ for each $t\in(0,T)$. Thus, by the It\^o-Wentzell formula (see e.g. \cite[Proposition 1]{M83}), we have
\begin{align*}
  v(t,x+Y_t) &= v(0,x) + \int_0^t \int_{\bR^d} y^{(\alpha)} \cdot \nabla v(s,x+Y_{s-})\, q(dsdy)
  \\
  &\quad + \int_0^t\int_{\bR^d} \left(y-y^{(\alpha)}\right) \cdot \nabla v(s,x+Y_{s-})\, p(dsdy)
  \\
  &\quad + \sum_{s\leq t}\left[ v(s,x+Y_s) - v(s,x+Y_{s-}) - \Delta Y_s \cdot \nabla v(s,x+Y_s) \right]
  \\
  &\quad + \int_0^t \partial_s v(s,x+Y_s) \,ds,
\end{align*}
where $\Delta Y_s:= Y_s - Y_{s-}$. Thus,
\begin{align} \label{eq9031702}
  &v(t,x+Y_t) \nonumber
  \\
  &= u_0(x) + \int_0^t \int_{\bR^d} y^{(\alpha)} \cdot \nabla v(s,x+Y_{s-})\, q(dsdy) \nonumber
  \\
  &\quad + \int_0^t\int_{\bR^d} \left(y-y^{(\alpha)}\right) \cdot \nabla v(s,x+Y_{s-})\, p(dsdy) \nonumber
  \\
  &\quad + \int_0^t\int_{\bR^d} \left[v(s,x+Y_{s-}+y) - v(s,x+Y_{s-}) - y \cdot \nabla v(s,x+Y_s) \right]\, p(dsdy) \nonumber
  \\
  &\quad + \int_0^t \left(L^{\nu^{(1)}} v(s,x+Y_s)+f(s,x)\right) \,ds.
\end{align}
Since $q(dsdy)$ is a martingale measure, by taking expectation of both sides of \eqref{eq9031702}, for $u(t,x):=\bE[v(t,x+Y_t)]$,
\begin{align*}
  u(t,x)&=u_0(x)+\int_0^t \int_{\bR^d} \left[u(s,x+y)-u(s,x)-y^{(\alpha)}\cdot \nabla u(s,x) \right] \left(\nu_t-\nu^{(1)}\right)(dy) \,ds
  \\
  &\quad+\int_0^t \left(L^{\nu^{(1)}} u(s,x)+f(s,x)\right) \,ds
  \\
  &=u_0(x)+ \int_0^t \left(L_t u(s,x)+f(s,x)\right) \,ds.
\end{align*}
Since $v(t,\cdot)\in C^2(\bR^d)$, we have $u(t,\cdot)\in C^2(\bR^d)$, and thus $u$ is a solution to \eqref{eq_whole}. Moreover, by using the Minkowski inequality, \eqref{ineq_whole} easily follows from the one for $v$.

For general $f$ and $u_0$, the desired result can be obtained from the denseness of $C_c^\infty$ functions in the spaces $L_p((0,T);L_p)$ and $B_{p,p}^{\alpha-\alpha/p}$, \eqref{eq9051621}, and the continuity of $L_t$. Here, we note that under Assumptions \ref{assum_whole} and \ref{nu_stable}, the constant $N$ depends only on $d,p,\alpha,\lambda,\Lambda$, and $N_0$.

\textbf{4.} Lastly, we show the uniqueness. Suppose that $u$ is a solution to \eqref{eq_whole} with $u_0=0$ and $f=0$.
Let $t_0\in (0,T)$, $v_0\in C_c^\infty(\bR^d)$, take a solution $v\in L_{p'}((0,t_0);H_{p'}^{\alpha}) \cap L_\infty((0,t_0);L_{p'})$ to the adjoint backward equation
  \begin{equation*}
\begin{cases}
\partial_t v(t,x)=-L_{t}^*v(t,x),\quad &(t,x)\in(0,t_0)\times \bR^d,
\\
v(t_0,x)=v_0(x),\quad & x\in \bR^d,
\end{cases}
\end{equation*}
where $p'=(p-1)/p$.
Then, for any $t\in(0,t_0)$ and $x\in \bR^d$,
\begin{equation*}
  u^\varepsilon(t,x)=\int_0^t L_s u^\varepsilon(s,x)\,ds,
\end{equation*}
and
\begin{align*}
  v^\varepsilon(t,x)&=v_0^\varepsilon(x)+\int_{t_0-t}^{t_0} L_{s}^* v^\varepsilon(s,x)\,ds.
\end{align*}
Thus, $u,v\in L_\infty((0,T);L_p)$ implies that, for each $x\in \bR^d$, both $u^\varepsilon(\cdot,x)$ and $v^\varepsilon(\cdot,x)$ are Lipschitz continuous in the time variable. Hence, by integration by parts,
\begin{align*}
  u^\varepsilon(t_0,x)v^\varepsilon_0(x) - u^\varepsilon(0,x)v^\varepsilon(0,x) &= \int_0^{t_0} \partial_s \left(u^\varepsilon(\cdot,x) v^\varepsilon(\cdot,x)\right) \,ds
  \\
  &=\int_0^{t_0} \left(L_s u^\varepsilon(s,x) v^\varepsilon(s,x) - u^\varepsilon(s,x) L_s^* v^\varepsilon(s,x) \right) \,ds.
\end{align*}
Thus, if we integrate both sides of the above over $\bR^d$, then we have
\begin{align*}
  \int_{\bR^d} u^\varepsilon(t_0,x)v^\varepsilon_0(x)\,dx =0.
\end{align*}
Letting $\varepsilon\downarrow0$, since $v_0\in C_c^\infty(\bR^d)$ is arbitrary, $u(t_0,\cdot)=0$.
 The lemma is proved.
\end{proof}

\section*{Acknowledgement}

The authors would like to thank the referees for their careful reading and very useful comments.

% \vspace{3em}
% \textbf{Declaration of interest}

% Declarations of interest: none

\end{document}